\documentclass[12pt]{amsart}
\usepackage{}
\usepackage{tikz}
  \usetikzlibrary{matrix,arrows,positioning,shapes.geometric}
\usepackage{amsmath}
\usepackage{amsfonts}
\usepackage{amssymb}
\usepackage[all]{xy}           %xypic macro for latex2.09

\usepackage{bbding}
\usepackage{txfonts}
\usepackage{amscd}

\usepackage[shortlabels]{enumitem}
\usepackage{ifpdf}
\ifpdf
  \usepackage[colorlinks,final,backref=page,hyperindex]{hyperref}
\else
  \usepackage[colorlinks,final,backref=page,hyperindex,hypertex]{hyperref}
\fi
\usepackage{tikz}
\usepackage[active]{srcltx}

\makeatletter

\newtheorem{thm}{Theorem}[section]
\newtheorem{lem}[thm]{Lemma}

\newtheorem{pro}[thm]{Proposition}

\newtheorem{defi}[thm]{Definition}

\newcommand {\emptycomment}[1]{}

\newcommand{\lon }{\,\rightarrow\,}
\newcommand{\be }{\begin{equation}}
\newcommand{\ee }{\end{equation}}

\newcommand{\g}{\mathfrak g}

\newcommand{\huaG}{\mathcal{G}}

\newcommand{\huaH}{\mathcal{H}}

\newcommand{\Id}{\rm{Id}}

\newcommand{\br}[1]{   [ \cdot,    \cdot  ]   }

\newcommand{\dM}{\mathrm{d}}

\newcommand{\Hom}{\mathrm{Hom}}

\newcommand{\gl}{\mathfrak {gl}}

\newcommand{\sgn}{\mathrm{sgn}}

\begin{document}

\title[]{Deformations and abelian extensions of compatible pre-Lie algebras}

\author{Shanshan Liu}
\address{School of Mathematics and Statistics, Northeast Normal University, Changchun 130024, Jilin, China}
\email{shanshanmath@163.com}

\author{Liangyun Chen*}
\address{School of Mathematics and Statistics, Northeast Normal University, Changchun 130024, China}
\email{chenly640@nenu.edu.cn}

%\date{\today}

\begin{abstract}
In this paper,  we first give the notation of a compatible pre-Lie algebra and its representation.  We study the relation between compatible Lie algebras and  compatible pre-Lie algebras. We also construct a new bidifferential graded Lie algebra whose Maurer-Cartan elements are compatible pre-Lie structures. We give the bidifferential graded Lie algebra which controls deformations of a compatible pre-Lie algebra. Then, we introduce a cohomology of a compatible pre-Lie algebra with coefficients in itself. We study infinitesimal deformations of compatible pre-Lie algebras and show that equivalent infinitesimal deformations are in the same second cohomology group. We further give the notion of a Nijenhuis operator on a compatible pre-Lie algebra. We study formal deformations of compatible pre-Lie algebras. If the second cohomology group $\huaH^2(\g;\g)$ is trivial, then the compatible pre-Lie algebra is rigid. Finally, we give a cohomology of a compatible pre-Lie algebra with coefficients in arbitrary representation and study abelian extensions of compatible pre-Lie algebras using this cohomology. We show that abelian extensions are classified by the second cohomology group.
\end{abstract}

\renewcommand{\thefootnote}{\fnsymbol{footnote}}
\footnote[0]{* Corresponding author}

\maketitle
\vspace{-5mm}
\tableofcontents

\allowdisplaybreaks

%\end{document}

\section{Introduction}
The notion  of a pre-Lie algebra (also called  left-symmetric algebras, quasi-associative algebras, Vinberg algebras and so on) has been introduced independently  by M. Gerstenhaber in  deformation theory  of rings and algebras \cite{Gerstenhaber63}. Pre-Lie algebra arose from the study of affine manifolds and affine structures on Lie group \cite{JLK}, homogeneous convex  cones \cite{Vinberg}. Its defining identity is weaker than associativity. This algebraic structure describes some properties of cochains space in Hochschild cohomology of an associative algebra, rooted trees and vector fields on affine spaces.
 Moreover, it is playing an increasing role in algebra, geometry and physics due to their applications in nonassociative algebras, combinatorics,  numerical Analysis and quantum field theory, see also in \cite{BD,Bai,Bai2,CK}. There is a close relationship between pre-Lie algebras and Lie algebras: a pre-Lie algebra $(\g,\cdot)$ gives rise to a Lie algebra $(\g,[\cdot,\cdot]_C)$ via the commutator bracket, which is called the subadjacent Lie algebra and denoted by ${\g}^C$. Furthermore, the map $L:\g\longrightarrow\gl(\g)$, defined by $L_xy=x\cdot y$ for all $x,y\in \g$, gives rise to a representation of the subadjacent Lie algebra ${\g}^C$ on $\g$.

Compatible algebraic structures refer to two algebraic structures of the same kind in a linear category such that any linear combination of multiplications corresponding to these two algebraic structures still defines the same kind of algebraic structure. Compatible algebraic structures have been widely studied in mathematics and mathematical physics. Golubchik and Sokolov studied compatible Lie algebras with the background of integrable equations \cite{GZI1}, classical Yang-Baxter
equations  \cite{GZI2}, loop algebras over Lie algebras  \cite{GZI3}. Compatible Lie algebras are also related to elliptic theta functions \cite{OAV}. Classification, operads and bialgebra theory of compatible Lie algebras were also studied in \cite{PA,SH,WM}. Recently, in \cite{Liu2}, deformation theory and the cohomology theory of compatible Lie algebras were established by using the theory of bidifferential graded Lie algebras. Using similar ideas, compatible $L_{\infty}$-algebras were studied in \cite{DA}, compatible associative algebras were studied in \cite{CT} and compatible $3$-Lie algebras were studied in \cite{HS}.

The purpose of the paper is to study the cohomology of a compatible pre-Lie algebra and its applications. We construct a bidifferential graded Lie algebra whose Maurer-Cartan elements are compatible pre-Lie structures. Then, we introduce a cohomology of a compatible pre-Lie algebra with coefficients in itself. As applications, we study formal deformations and infinitesimal deformations of compatible pre-Lie algebras and give the notion of a Nijenhuis operator on a compatible pre-Lie algebra. We give a compatible pre-Lie algebra $(\g,\pi_1,\pi_2)$ and its representation $(V,\rho,\mu,\tilde{\rho},\tilde{\mu})$ and construct a bidifferential graded Lie algebra whose Maurer-Cartan elements is $(\pi_1+\rho+\mu,\pi_2+\tilde{\rho}+\tilde{\mu})$. Then we give a cohomology of a compatible pre-Lie algebra with coefficients in arbitrary representation. As applications, we study abelian extensions of compatible pre-Lie algebras.

The paper is organized as follows. In Section \ref{sec:Maurer-Cartan}, we give the notation of a compatible pre-Lie algebra and its representation. We study the relation between compatible Lie algebras and  compatible pre-Lie algebras. We recall the notion of bidifferential graded Lie algebra and construct a new bidifferential graded Lie algebra whose Maurer-Cartan elements are compatible pre-Lie structures. Furthermore, we  give the bidifferential graded Lie algebra which controls deformations of a compatible pre-Lie algebra. In Section \ref{sec:infinitesimal-deformation}, we introduce a cohomology of a compatible pre-Lie algebra with coefficients in itself. Using this cohomology, we study infinitesimal deformations of compatible pre-Lie algebras and show that equivalent infinitesimal deformations are in the same second cohomology group. We give the notion of a Nijenhuis operator on a compatible pre-Lie algebra. We show that a trivial deformation gives rise to a Nijenhuis operator. Conversely, a Nijenhuis operator gives rise to a trivial deformation. We study formal deformations of compatible pre-Lie algebras. If the second cohomology group $\huaH^2(\g;\g)$ is trivial, then the compatible pre-Lie algebra is rigid.  In Section \ref{sec:extension}, we give a compatible pre-Lie algebra $(\g,\pi_1,\pi_2)$ and its representation $(V,\rho,\mu,\tilde{\rho},\tilde{\mu})$ and construct a bidifferential graded Lie algebra whose Maurer-Cartan elements is $(\pi_1+\rho+\mu,\pi_2+\tilde{\rho}+\tilde{\mu})$. Using this bidifferential graded Lie algebra, we give a cohomology of a compatible pre-Lie algebra with coefficients in arbitrary representation. We study abelian extensions of compatible pre-Lie algebras using this cohomological approach and show that abelian extensions are classified by the second cohomology group.
\vspace{2mm}
\noindent
\section{Maurer-Cartan characterizations of compatible pre-Lie algebras}\label{sec:Maurer-Cartan}
In this section, first, we give the notation of a compatible pre-Lie algebra and its representation. Then, we study the relation between compatible Lie algebras and  compatible pre-Lie algebras. Finally, we construct a new bidifferential graded Lie algebra whose Maurer-Cartan elements are compatible pre-Lie structures. We  give the bidifferential graded Lie algebra which controls deformations of a compatible pre-Lie algebra.
\begin{defi}{\rm(\cite{BD})}
A {\bf pre-Lie algebra} $(\g,\cdot)$ is a vector space $\g$ equipped with a bilinear product $\cdot:\g\otimes \g\longrightarrow \g$, such that for all $x,y,z\in \g$, the following equality is satisfied:
\begin{equation*}
(x\cdot y)\cdot z-x\cdot (y\cdot z)=(y\cdot x)\cdot z-y\cdot (x\cdot z).
\end{equation*}
\end{defi}
Let $(\g,\cdot)$ be a pre-Lie algebra. The commutator $[x,y]_C=x\cdot y-y\cdot x$ gives a Lie algebra $(\g,[\cdot,\cdot]_C)$, which is denoted by ${\g}^C$ and called  the {\bf sub-adjacent Lie algebra} of $(\g,\cdot)$.
\begin{defi}{\rm(\cite{Bai})}
 A {\bf representation} of a pre-Lie algebra $(\g,\cdot)$ on a vector space $V$ consists of a pair $(\rho,\mu)$, where $\rho:\g\longrightarrow \gl(V)$ is a representation of the sub-adjacent Lie algebra ${\g}^C$ on $V$, and $\mu:\g\longrightarrow \gl(V)$ is a linear map, such that for all $x,y\in \g$:
\begin{equation*}
\mu(y)\circ\mu(x)-\mu(x\cdot y)=\mu(y)\circ\rho(x)-\rho(x)\circ\mu(y).
\end{equation*}
\end{defi}
We denote a representation of a pre-Lie algebra $(\g,\cdot)$ by $(V,\rho,\mu)$. Furthermore, let $L,R:\g\longrightarrow \gl(\g)$ be linear maps, where $L_xy=x\cdot y, R_xy=y\cdot x$. Then $(\g,L,R)$ is also a representation, which is called the regular representation.

A permutation $\sigma\in \mathbb{S}_n$ is called an $(i,n-i)$-unshuffle if $\sigma(1)<\dots<\sigma(i)$ and $\sigma(i+1)<\dots<\sigma(n)$. If $i=0$ and $i=n$, we assume $\sigma=\Id$. The set of all $(i,n-i)$-unshuffles will be denoted by $\mathbb{S}_{(i,n-i)}$. The notion of an $(i_1,\dots,i_k)$-unshuffle and the set $\mathbb{S}_{(i_1,\dots,i_k)}$ are defined similarly.

Let $\g$ be a vector space. We denote $C^n(\g;\g)=\Hom(\wedge^{n-1}\g\otimes \g,\g)$ and consider the graded vector space $C^*(\g;\g)=\oplus_{n=1}^{+\infty}C^n(\g;\g)=\oplus_{n=1}^{+\infty}\Hom(\wedge^{n-1}\g\otimes \g,\g)$. It was shown in \cite{FC,AN,QW} that $C^*(\g;\g)$ equipped with the Matsushima-Nijenhuis bracket
\begin{equation*}
[P,Q]^{MN}=P\circ Q-(-1)^{pq}Q\circ P,\quad \forall P\in C^{p+1}(\g;\g), Q\in C^{q+1}(\g;\g)
\end{equation*}
ia a graded Lie algebra, where $P\circ Q\in C^{p+q+1}(\g;\g)$ is defined by
\begin{eqnarray*}
 &&P\circ Q(x_1,\dots,x_{p+q+1})\\
&=&\sum_{\sigma\in \mathbb{S}(q,1,p-1)}\sgn(\sigma)P(Q(x_{\sigma(1)},\dots,x_{\sigma(q)},x_{\sigma(q+1)}),x_{\sigma(q+2)},\dots,x_{\sigma(p+q)},x_{p+q+1})\\
&&+(-1)^{pq}\sum_{\sigma\in \mathbb{S}(p,q)}\sgn(\sigma)P(x_{\sigma(1)},\dots,x_{\sigma(p)},Q(x_{\sigma(p+1)},\dots,x_{\sigma(p+q)},x_{p+q+1})).
\end{eqnarray*}
In particular, $\pi\in \Hom(\otimes^2\g,\g)$ defines a pre-Lie algebra if and only if $[\pi,\pi]^{MN}=0$. If $\pi$ is a pre-Lie algebra structure, then $d_{\pi}:=[\pi,\cdot]^{MN}$ is a graded derivation of the graded Lie algebra $(C^*(\g;\g),[\cdot,\cdot]^{MN})$ satisfying $d_{\pi}\circ d_{\pi}=0$, so that $(C^*(\g;\g),[\cdot,\cdot]^{MN},d_{\pi})$ becomes a differential graded Lie algebra.

\begin{defi}
A {\bf compatible pre-Lie algebra} is a triple $(\g,\cdot,\ast)$, where $\g$ is a vector space, $``\cdot$'' and $``\ast$'' are pre-Lie structures on $\g$, such that for all $x,y,z\in \g$, the following equality is satisfied:
\begin{equation}\label{compatible-pre-Lie}
(x\ast y)\cdot z+(x\cdot y)\ast z-x\cdot(y\ast z)-x\ast(y\cdot z)-(y\ast x)\cdot z-(y\cdot x)\ast z+y\cdot(x\ast z)+y\ast(x\cdot z)=0.
\end{equation}
\end{defi}
\begin{pro}\label{linear}
A triple $(\g,\cdot,\ast)$ is a compatible pre-Lie algebra if and only if $``\cdot$'' and $``\ast$'' are pre-Lie structures on $\g$, such that for all $k_1,k_2\in \mathbb{K}$, the following bilinear operation
\begin{equation}\label{compatible-product}
x\diamond y=k_1 x\cdot y+k_2x\ast y,\quad \forall x,y\in\g.
\end{equation}
defines a pre-Lie algebra structure on $\g$.
\end{pro}
\begin{proof}
It is straightforward.
\end{proof}

\begin{defi}
Let $(\g,\cdot,\ast)$ and $(\g',\cdot',\ast')$ be two compatible pre-Lie algebras. A {\bf homomorphism} $\varphi:(\g,\cdot,\ast)\longrightarrow (\g',\cdot',\ast')$ is both a pre-Lie homomorphism from $(\g,\cdot)$ to $(\g',\cdot')$ and a pre-Lie homomorphism from $(\g,\ast)$ to $(\g',\ast')$.
\end{defi}

\begin{defi}
 A {\bf representation} of a compatible pre-Lie algebra $(\g,\cdot,\ast)$ on a vector space $V$ consists of a quadruple $(\rho,\mu,\tilde{\rho},\tilde{\mu})$, where $(V,\rho,\mu)$ is a representation of the pre Lie algebra $(\g,\cdot)$ and $(V,\tilde{\rho},\tilde{\mu})$ is a representation of the pre Lie algebra $(\g,\ast)$, such that for all $x,y\in \g$:
\begin{eqnarray}
\label{rep-1}\rho(x\ast y)+\tilde{\rho}(x\cdot y)-\rho(x)\tilde{\rho}(y)-\tilde{\rho}(x)\rho(y)&=&\rho(y\ast x)+\tilde{\rho}(y\cdot x)-\rho(y)\tilde{\rho}(x)-\tilde{\rho}(y)\rho(x),\\
\label{rep-2}\mu(y)\tilde{\rho}(x)-\rho(x)\tilde{\mu}(y)-\mu(y)\tilde{\mu}(x)+\mu(x\ast y)&=&-\tilde{\mu}(y)\rho(x)+\tilde{\rho}(x)\mu(y)+\tilde{\mu}(y)\mu(x)-\tilde{\mu}(x\cdot y).
\end{eqnarray}
\end{defi}
We denote a representation of a compatible pre-Lie algebra $(\g,\cdot,\ast)$ by $(V,\rho,\mu,\tilde{\rho},\tilde{\mu})$. Furthermore, let $L,R,\tilde{L},\tilde{R}:\g\longrightarrow \gl(\g)$ be linear maps, where $L_xy=x\cdot y, R_xy=y\cdot x, \tilde{L}_xy=x\ast y, \tilde{R}_xy=y\ast x$. Then $(\g,L,R,\tilde{L},\tilde{R})$ is also a representation, which is called the {\bf regular representation}.

 We define two bilinear operations $\cdot_{\g\oplus V}:\otimes^2(\g\oplus V)\lon(\g\oplus V)$ and $\ast_{\g\oplus V}:\otimes^2(\g\oplus V)\lon(\g\oplus V)$ respectively by
\begin{eqnarray*}
(x+u)\cdot_{\g\oplus V} (y+v)&=&x\cdot y+\rho(x)(v)+\mu(y)(u),\quad \forall x,y \in \g,  u,v\in V,\\
(x+u)\ast_{\g\oplus V} (y+v)&=&x\ast y+\tilde{\rho}(x)(v)+\tilde{\mu}(y)(u),\quad \forall x,y \in \g,  u,v\in V.
\end{eqnarray*}
\begin{pro}
 With the above notation, $(\g\oplus V,\cdot_{\g\oplus V},\ast_{\g\oplus V})$ is a compatible pre-Lie algebra, which is denoted by $\g\ltimes_{(\rho,\mu,\tilde{\rho},\tilde{\mu})}V$ and called  the  {\bf semi-direct product} of the compatible pre-Lie algebra $(\g,\cdot,\ast)$ and the representation $(V,\rho,\mu,\tilde{\rho},\tilde{\mu})$.
\end{pro}
\begin{proof}
Obviously, $(\g\oplus V,\cdot_{\g\oplus V})$ and $(\g\oplus V,\ast_{\g\oplus V})$ are pre-Lie algebras. For all $x,y,z\in\g, u,v,w\in V$, by \eqref{compatible-pre-Lie}, \eqref{rep-1} and \eqref{rep-2}, we have
\begin{eqnarray*}
&&((x+u)\ast_{\g\oplus V} (y+v))\cdot_{\g\oplus V} (z+w)+((x+u)\cdot_{\g\oplus V} (y+v))\ast_{\g\oplus V} (z+w)\\
&&-(x+u)\cdot_{\g\oplus V}((y+v)\ast_{\g\oplus V} (z+w))-(x+u)\ast_{\g\oplus V}((y+v)\cdot_{\g\oplus V} (z+w))\\
&&-((y+v)\ast_{\g\oplus V} (x+u))\cdot_{\g\oplus V} (z+w)-((y+v)\cdot_{\g\oplus V} (x+u))\ast_{\g\oplus V} (z+w)\\
&&+(y+v)\cdot_{\g\oplus V}((x+u)\ast_{\g\oplus V} (z+w))+(y+v)\ast_{\g\oplus V}((x+u)\cdot_{\g\oplus V} (z+w))\\
&=&(x\ast y)\cdot z+\rho(x\ast y)w+\mu(z)\tilde{\rho}(x)v+\mu(z)\tilde{\mu}(y)u+(x\cdot y)\ast z+\tilde{\rho}(x\cdot y)w+\tilde{\mu}(z)\rho(x)v+\tilde{\mu}(z)\mu(y)u\\
&&-x\cdot(y\ast z)-\rho(x)\tilde{\rho}(y)w-\rho(x)\tilde{\mu}(z)v-\mu(y\ast z)u-x\ast(y\cdot z)-\tilde{\rho}(x)\rho(y)w-\tilde{\rho}(x)\mu(z)v-\tilde{\mu}(y\cdot z)u\\
&&-(y\ast x)\cdot z-\rho(y\ast x)w-\mu(z)\tilde{\rho}(y)u-\mu(z)\tilde{\mu}(x)v-(y\cdot x)\ast z-\tilde{\rho}(y\cdot x)w-\tilde{\mu}(z)\rho(y)u-\tilde{\mu}(z)\mu(x)v\\
&&+y\cdot(x\ast z)+\rho(y)\tilde{\rho}(x)w+\rho(y)\tilde{\mu}(z)u+\mu(x\ast z)v+y\ast(x\cdot z)+\tilde{\rho}(y)\rho(x)w+\tilde{\rho}(y)\mu(z)u+\tilde{\mu}(x\cdot z)v\\
&=&0.
\end{eqnarray*}
This finishes the proof.
\end{proof}
Now, we will give the relation between compatible Lie algebras and compatible pre-Lie algebras. First, we will recall the notation of a compatible Lie algebra and its representation.
\begin{defi}{\rm(\cite{GZI1,GZI2,OAV})}
A {\bf compatible Lie algebra} is a triple $(\g,[\cdot,\cdot],\{\cdot,\cdot\})$, where $\g$ is a vector space, $[\cdot,\cdot]$ and $\{\cdot,\cdot\}$ are Lie algebra structures on $\g$, such that for all $x,y,z\in \g$, the following equality is satisfied:
\begin{equation*}
[\{x,y\},z]+[\{y,z\},x]+[\{z,x\},y]+\{[x,y],z\}+\{[y,z],x\}+\{[z,x],y\}=0.
\end{equation*}
\end{defi}
\begin{defi}{\rm(\cite{WM})}
 A {\bf representation} of a compatible Lie algebra $(\g,[\cdot,\cdot],\{\cdot,\cdot\})$ on a vector space $V$ consists of a pair $(\rho,\mu)$, where $(V,\rho)$ is a representation of the Lie algebra $(\g,[\cdot,\cdot])$ and $(V,\mu)$ is a representation of the Lie algebra $(\g,\{\cdot,\cdot\})$ such that for all $x,y\in \g$:
\begin{equation*}
\rho(\{x,y\})+\mu([x,y])=[\rho(x),\mu(y)]-[\rho(y),\mu(x)].
\end{equation*}
\end{defi}
We denote a representation of a compatible Lie algebra $(\g,[\cdot,\cdot],\{\cdot,\cdot\})$ by $(V,\rho,\mu)$.
\begin{pro}
Let $(\g,\cdot,\ast)$ be a compatible pre-Lie algebra. Define two brackets $[\cdot,\cdot]$ and $\{\cdot,\cdot\}$ respectively by
\begin{equation*}
[x,y]=x\cdot y-y\cdot x,\quad \{x,y\}=x\ast y-y\ast x.
\end{equation*}
Then, $(\g,[\cdot,\cdot],\{\cdot,\cdot\})$ is a compatible Lie algebra, which is denoted by ${\g}^C$ and called  the {\bf sub-adjacent compatible Lie algebra} of $(\g,\cdot,\ast)$. Moreover, let $L,\tilde{L}$ be linear maps, where $L_xy=x\cdot y, \tilde{L}_xy=x\ast y$. Then $(V,L,\tilde{L})$ is a representation of the sub-adjacent compatible Lie algebra $(\g,[\cdot,\cdot],\{\cdot,\cdot\})$.
\end{pro}
\begin{proof}
Obviously, $(\g,[\cdot,\cdot])$ and $(\g,\{\cdot,\cdot\})$ are Lie algebras. For all $x,y,z\in \g$, by \eqref{compatible-pre-Lie}, we have
\begin{eqnarray*}
&&[\{x,y\},z]+[\{y,z\},x]+[\{z,x\},y]+\{[x,y],z\}+\{[y,z],x\}+\{[z,x],y\}\\
&=&[x\ast y-y\ast x,z]+[y\ast z-z\ast y,x]+[z\ast x-x\ast z,y]\\
&&+\{x\cdot y-y\cdot x,z\}+\{y\cdot z-z\cdot y,x\}+\{z\cdot x-x\cdot z,y\}\\
&=&(x\ast y)\cdot z-(y\ast x)\cdot z-z\cdot(x\ast y)+z\cdot(y\ast x)+(y\ast z)\cdot x-(z\ast y)\cdot x\\
&&-x\cdot(y\ast z)+x\cdot(z\ast y)+(z\ast x)\cdot y-(x\ast z)\cdot y-y\cdot(z\ast x)+y\cdot(x\ast z)\\
&&+(x\cdot y)\ast z-(y\cdot x)\ast z-z\ast(x\cdot y)+z\ast(y\cdot x)+(y\cdot z)\ast x-(z\cdot y)\ast x\\
&&-x\ast(y\cdot z)+x\ast(z\cdot y)+(z\cdot x)\ast y-(x\cdot z)\ast y-y\ast(z\cdot x)+y\ast(x\cdot z)\\
&=&0.
\end{eqnarray*}
Thus, $(\g,[\cdot,\cdot],\{\cdot,\cdot\})$ is a compatible Lie algebra.

Obviously, $(V,L)$ is a representation of $(\g,[\cdot,\cdot])$ and $(V,\tilde{L})$ is a representation of $(\g,\{\cdot,\cdot\})$. For all $x,y,z\in \g$, by \eqref{compatible-pre-Lie}, we have
 \begin{eqnarray*}
&&L_{\{x,y\}}z+\tilde{L}_{[x,y]}z-[L_x,\tilde{L}_y]z+[L_y,\tilde{L}_x]z\\
&=&\{x,y\}\cdot z+[x,y]\ast z-L_x\tilde{L}_yz+\tilde{L}_yL_xz+L_y\tilde{L}_xz-\tilde{L}_xL_yz\\
&=&(x\ast y)\cdot z-(y\ast x)\cdot z+(x\cdot y)\ast z-(y\cdot x)\ast z-x\cdot(y\ast z)+y\ast(x\cdot z)+y\cdot(x\ast z)-x\ast(y\cdot z)\\
&=&0,
\end{eqnarray*}
which implies that $$L_{\{x,y\}}+\tilde{L}_{[x,y]}=[L_x,\tilde{L}_y]-[L_y,\tilde{L}_x].$$ Thus, $(V,L,\tilde{L})$ is a representation of the sub-adjacent compatible Lie algebra $(\g,[\cdot,\cdot],\{\cdot,\cdot\})$.
\end{proof}

\begin{pro}
Let $(V,\rho,\mu,\tilde{\rho},\tilde{\mu})$ be a representation of a compatible pre-Lie algebra $(\g,\cdot,\ast)$. Then $(V,\rho-\mu,\tilde{\rho}-\tilde{\mu})$ is a representation of the sub-adjacent compatible Lie algebra $(\g,[\cdot,\cdot],\{\cdot,\cdot\})$.
\end{pro}
\begin{proof}
 Obviously, $(V,\rho-\mu)$ is a representation of the Lie algebra $(\g,[\cdot,\cdot])$ and $(V,\tilde{\rho}-\tilde{\mu})$ is a representation of the Lie algebra $(\g,\{\cdot,\cdot\})$. For all $x,y\in \g, u\in V$, by \eqref{rep-1} and \eqref{rep-2}, we have
 \begin{eqnarray*}
&&(\rho-\mu)(\{x,y\})+(\tilde{\rho}-\tilde{\mu})([x,y])-[(\rho-\mu)(x),(\tilde{\rho}-\tilde{\mu})(y)]+[(\rho-\mu)(y),(\tilde{\rho}-\tilde{\mu})(x)]\\
&=&\rho\{x,y\}-\mu\{x,y\}+\tilde{\rho}[x,y]-\tilde{\mu}[x,y]-[\rho(x),\tilde{\rho}(y)]+[\rho(x),\tilde{\mu}(y)]+[\mu(x),\tilde{\rho}(y)]\\
&&-[\mu(x),\tilde{\mu}(y)]+[\rho(y),\tilde{\rho}(x)]-[\rho(y),\tilde{\mu}(x)]-[\mu(y),\tilde{\rho}(x)]+[\mu(y),\tilde{\mu}(x)]\\
&=&\rho(x\ast y)-\rho(y\ast x)-\mu(x\ast y)+\mu(y\ast x)+\tilde{\rho}(x\cdot y)-\tilde{\rho}(y\cdot x)-\tilde{\mu}(x\cdot y)+\tilde{\mu}(y\cdot x)\\
&&-\rho(x)\tilde{\rho}(y)+\tilde{\rho}(y)\rho(x)+\rho(x)\tilde{\mu}(y)-\tilde{\mu}(y)\rho(x)+\mu(x)\tilde{\rho}(y)-\tilde{\rho}(y)\mu(x)-\mu(x)\tilde{\mu}(y)+\tilde{\mu}(y)\mu(x)\\
&&+\rho(y)\tilde{\rho}(x)-\tilde{\rho}(x)\rho(y)-\rho(y)\tilde{\mu}(x)+\tilde{\mu}(x)\rho(y)-\mu(y)\tilde{\rho}(x)+\tilde{\rho}(x)\mu(y)+\mu(y)\tilde{\mu}(x)-\tilde{\mu}(x)\mu(y)\\
&=&0,
\end{eqnarray*}
 which implies that $$(\rho-\mu)(\{x,y\})+(\tilde{\rho}-\tilde{\mu})([x,y])=[(\rho-\mu)(x),(\tilde{\rho}-\tilde{\mu})(y)]-[(\rho-\mu)(y),(\tilde{\rho}-\tilde{\mu})(x)].$$ This finishes the proof.
\end{proof}

\begin{defi}{\rm(\cite{Liu2})}
Let $(\huaG,[\cdot,\cdot],\dM_1)$ and $(\huaG,[\cdot,\cdot],\dM_2)$ be two differential graded Lie algebras, where $\huaG=\oplus_{i=0}^{\infty}{\g}_i$. We call the quadruple $(\huaG,[\cdot,\cdot],\dM_1,\dM_2)$ a {\bf bidifferential graded Lie algebra} if $\dM_1$ and $\dM_2$ satisfy
\begin{equation}
\dM_1\circ \dM_2+\dM_2\circ \dM_1=0.
\end{equation}
\end{defi}
\begin{defi}{\rm(\cite{Liu2})}
Let $(\huaG,[\cdot,\cdot],\dM_1,\dM_2)$ be a bidifferential graded Lie algebra. A pair $(P_1,P_2)\in {\g}_1\oplus {\g}_1$  is called a {\bf Maurer-Cartan element} of $(\huaG,[\cdot,\cdot],\dM_1,\dM_2)$ if $P_1$ and $P_2$ are Maurer-Cartan elements of the differential graded Lie algebra $(\huaG,[\cdot,\cdot],\dM_1)$ and $(\huaG,[\cdot,\cdot],\dM_2)$ respectively, such that
\begin{equation}
\dM_1P_2 +\dM_2P_1+[P_1,P_2]=0.
\end{equation}
\end{defi}
Let $(\huaG,[\cdot,\cdot])$ be a graded Lie algebra. It is obviously that $(\huaG,[\cdot,\cdot],\dM_1=0,\dM_2=0)$ is a bidifferential graded Lie algebra. Consider the graded Lie algebra $(C^*(\g;\g),[\cdot,\cdot]^{MN})$, we obtain the following main result.
\begin{thm}\label{thm-MC}
Let $\g$ be a vector space and $\pi_1,\pi_2\in \Hom(\otimes^2\g,\g)$. Then $(\g,\pi_1,\pi_2)$ is a compatible pre-Lie algebra if and only if $(\pi_1,\pi_2)$ is a Maurer-Cartan element of the bidifferential graded Lie algebra $(C^*(\g;\g),[\cdot,\cdot]^{MN},\dM_1=0,\dM_2=0)$.
\end{thm}
\begin{proof}
Obviously, $(\g,\pi_1)$ and $(\g,\pi_2)$ are pre-Lie algebra if and only if
\begin{equation*}
[\pi_1,\pi_1]^{MN}=0,\quad [\pi_2,\pi_2]^{MN}=0.
\end{equation*}
For all $x,y,z\in\g$, we have
\begin{eqnarray*}
[\pi_1,\pi_2]^{MN}(x,y,z)&=&\pi_1(\pi_2(x,y),z)-\pi_1(\pi_2(y,x),z)-\pi_1(x,\pi_2(y,z))+\pi_1(y,\pi_2(x,z))\\
&&+\pi_2(\pi_1(x,y),z)-\pi_2(\pi_1(y,x),z)-\pi_2(x,\pi_1(y,z))+\pi_2(y,\pi_1(x,z)),
\end{eqnarray*}
which implies that equation \eqref{compatible-pre-Lie} is equivalent to $[\pi_1,\pi_2]^{MN}=0$. Thus, $(\g,\pi_1,\pi_2)$ is a compatible pre-Lie algebra if and only if $(\pi_1,\pi_2)$ is a Maurer-Cartan element of the bidifferential graded Lie algebra $(C^*(\g;\g),[\cdot,\cdot]^{MN},\dM_1=0,\dM_2=0)$.
\end{proof}
Now, we give a new bidifferential graded Lie algebra that controls deformations of a compatible pre-Lie algebra.
\begin{pro}
Let $(\g,\pi_1,\pi_2)$ be a compatible pre-Lie algebra. Then $(C^*(\g;\g),[\cdot,\cdot]^{MN},\dM_{\pi_1},\dM_{\pi_2})$ is a bidifferential graded Lie algebra. Moreover, $(\g,\pi_1+\pi_1',\pi_2+\pi_2')$ is a compatible pre-Lie algebra for all $\pi'_1,\pi'_2\in \Hom(\otimes^2\g,\g)$ if and only if $(\pi'_1,\pi'_2)$ is a Maurer-Cartan element of the bidifferential graded Lie algebra $(C^*(\g;\g),[\cdot,\cdot]^{MN},\dM_{\pi_1},\dM_{\pi_2})$.
\end{pro}
\begin{proof}
Since $(\g,\pi_1,\pi_2)$ is a compatible pre-Lie algebra, by Theorem \ref{thm-MC}, $(\pi_1,\pi_2)$ is a Maurer-Cartan element of the bidifferential graded Lie algebra $(C^*(\g;\g),[\cdot,\cdot]^{MN},\dM_1=0,\dM_2=0)$. Thus, we have
\begin{equation*}
[\pi_1,\pi_1]^{MN}=0,\quad [\pi_2,\pi_2]^{MN}=0, \quad [\pi_1,\pi_2]^{MN}=0.
\end{equation*}
Thus, $(C^*(\g;\g),[\cdot,\cdot]^{MN},\dM_{\pi_1})$ and $(C^*(\g;\g),[\cdot,\cdot]^{MN},\dM_{\pi_2})$ are differential graded Lie algebras. For all $P\in C^{p+1}(\g,\g)$, by graded Jacobi identity, we have
\begin{equation*}
\dM_{\pi_1}(\dM_{\pi_2}P)+\dM_{\pi_2}(\dM_{\pi_1}P)=[\pi_1,[\pi_2,P]^{MN}]^{MN}+[\pi_2,[\pi_1,P]^{MN}]^{MN}=[[\pi_1,\pi_2]^{MN},p]^{MN}=0,
\end{equation*}
which implies that $\dM_{\pi_1}\circ \dM_{\pi_2}+\dM_{\pi_2}\circ\dM_{\pi_1}=0$.
Thus, $(C^*(\g;\g),[\cdot,\cdot]^{MN},\dM_{\pi_1},\dM_{\pi_2})$ is a bidifferential graded Lie algebra.

If $(\g,\pi_1+\pi_1',\pi_2+\pi_2')$ is a compatible pre-Lie algebra, by Theorem \ref{thm-MC}, $(\pi_1+\pi_1',\pi_2+\pi_2')$ is a Maurer-Cartan element of the bidifferential graded Lie algebra $(C^*(\g;\g),[\cdot,\cdot]^{MN},\dM_1=0,\dM_2=0)$. Thus, we have
\begin{eqnarray}
\label{MC-1}[\pi_1+\pi_1',\pi_1+\pi_1']^{MN}&=&0,\\
\label{MC-2}[\pi_2+\pi_2',\pi_2+\pi_2']^{MN}&=&0,\\
\label{MC-3}[\pi_1+\pi_1',\pi_2+\pi_2']^{MN}&=&0,
\end{eqnarray}
By \eqref{MC-1}, \eqref{MC-2} and  \eqref{MC-3}, we have
\begin{eqnarray*}
\dM_{\pi_1}\pi_1'+\frac{1}{2}[\pi_1',\pi_1']^{MN}&=&0,\\
\dM_{\pi_2}\pi_2'+\frac{1}{2}[\pi_2',\pi_2']^{MN}&=&0,\\
\dM_{\pi_1}\pi_2'+\dM_{\pi_2}\pi_1'+[\pi_1',\pi_2']^{MN}&=&0.
\end{eqnarray*}
Thus, $(\pi'_1,\pi'_2)$ is a Maurer-Cartan element of  $(C^*(\g;\g),[\cdot,\cdot]^{MN},\dM_{\pi_1},\dM_{\pi_2})$. The converse part can be proved similarly. We omit details. The proof is finished.
\end{proof}
\section{Formal deformations of compatible pre-Lie algebras}\label{sec:infinitesimal-deformation}
In this section, first, we introduce a cohomology of a compatible pre-Lie algebra with coefficients in itself. Then, we study infinitesimal deformations of compatible pre-Lie algebras using this cohomology, we show that equivalent infinitesimal deformations are in the same second cohomology group. We give the notion of a Nijenhuis operator on a compatible pre-Lie algebra and show that a Nijenhuis operator gives rise to a trivial deformation. Finally, we study formal deformations of compatible pre-Lie algebras. If the second cohomology group $\huaH^2(\g;\g)$ is trivial, then the compatible pre-Lie algebra is rigid.
\subsection{Cohomologies of compatible pre-Lie algebras}
\
\newline
\indent
Let $(\g,\pi)$ be a pre-Lie algebra, where $\pi(x,y)=x\cdot y$. Because of the graded Jacobi identity, we define a coboundary operator $\delta_{\pi}:C^n(\g,\g)\longrightarrow C^{n+1}(\g,\g)$ by
\begin{equation*}
\delta_{\pi} f=(-1)^{n-1}[\pi,f]^{MN},\quad \forall f\in C^n(\g;\g).
\end{equation*}
Thus, we obtain a cochain complex $(\oplus_{n=1}^{+\infty}C^n(\g;\g),\delta_{\pi})$. More precisely, for all $x_1,\dots,x_{n+1}\in \g$, we have
\begin{eqnarray*}
(\delta_{\pi} f)(x_1,\dots,x_{n+1})&=&\sum_{i=1}^n(-1)^{i+1}x_i\cdot f(x_1,\dots,\hat{x_i},\dots,x_{n+1})\\
&&+\sum_{i=1}^n(-1)^{i+1}f(x_1,\dots,\hat{x_i},\dots,x_n,x_i)\cdot x_{n+1}\\
&&-\sum_{i=1}^n(-1)^{i+1}f(x_1,\dots,\hat{x_i}\dots,x_n,x_i\cdot x_{n+1})\\
&&+\sum_{1\leq i<j\leq n}(-1)^{i+j} f([x_i,x_j]_C,x_1,\dots,\hat{x_i},\dots,\hat{x_j},\dots,x_{n+1}),
\end{eqnarray*}
which is a coboundary operator of pre-Lie algebra $(\g,\pi)$ with coefficients in the regular representation $(V,L,R)$. We can see more details in \cite{DA}.

Let $(\g,\cdot,\ast)$ be a compatible pre-Lie algebra with $\pi_1(x,y)=x\cdot y$ , $\pi_2(x,y)=x\ast y$. We define the set of $n$-cochains $(n\geq1)$ by
\begin{equation*}
\mathfrak{C}^n(\g;\g)=\underbrace{C^n(\g;\g)\oplus C^n(\g;\g)\oplus\cdots \oplus C^n(\g;\g)}_{n~copies}.
\end{equation*}
Define the operator $\delta:\mathfrak{C}^n(\g;\g)\longrightarrow \mathfrak{C}^{n+1}(\g;\g)$ by
\begin{eqnarray*}
\delta^1 f&=&(\delta_{\pi_1}f,\delta_{\pi_2}f),\quad
 \forall f \in  \Hom(\g,\g), n=1,\\
\delta^n (f_1,\dots,f_n)&=&(\delta_{\pi_1}f_1,\dots, \underbrace{\delta_{\pi_2} f_{i-1}+\delta_{\pi_1} f_i}_{i},\dots,\delta_{\pi_2} f_n),\quad
 \forall (f_1,\dots,f_n) \in \mathfrak{C}^n(\g,\g), 2\leq i\leq n.
\end{eqnarray*}
%The following diagram will explain the above operators:

\begin{thm}\label{coboundary-operaor}
The operator  $\delta:\mathfrak{C}^n(\g;\g)\longrightarrow \mathfrak{C}^{n+1}(\g;\g)$ defined as above satisfies $\delta\circ\delta=0$.
\end{thm}
\begin{proof}
By Theorem \ref{thm-MC}, we obtain that $(\pi_1,\pi_2)$ is a Maurer-Cartan element of the bidifferential graded Lie algebra $(C^*(\g;\g),[\cdot,\cdot]^{MN},\dM_1=0,\dM_2=0)$. Thus, by the fact that $[\pi_1,\pi_1]^{MN}=[\pi_2,\pi_2]^{MN}=[\pi_1,\pi_2]^{MN}=0$ and the graded Jacobi identity, for all $f \in  \Hom(\g,\g)$, we have
 \begin{eqnarray*}
&&\delta^2(\delta^1f)\\
&=&\delta^2([\pi_1,f]^{MN},[\pi_2,f]^{MN})\\
&=&-([\pi_1,[\pi_1,f]^{MN}]^{MN},[\pi_2,[\pi_1,f]^{MN}]^{MN}+[\pi_1,[\pi_2,f]^{MN}]^{MN},[\pi_2,[\pi_2,f]^{MN}]^{MN})\\
&=&-(\frac{1}{2}[[\pi_1,\pi_1]^{MN},f]^{MN},[[\pi_1,\pi_2]^{MN},f]^{MN},\frac{1}{2}[[\pi_2,\pi_2]^{MN},f]^{MN})\\
&=&(0,0,0).
\end{eqnarray*}
By $[\pi_1,\pi_1]^{MN}=[\pi_2,\pi_2]^{MN}=[\pi_1,\pi_2]^{MN}=0$ and the graded Jacobi identity, for all $(f_1,\dots,f_n) \in \mathfrak{C}^n(\g,\g), 2\leq i\leq n$, we have
 \begin{eqnarray*}
&&\delta^{n+1}\delta^n(f_1,\dots,f_n)\\
&=&(-1)^{n-1}\delta^{n+1}([\pi_1,f_1]^{MN},\dots, \underbrace{[\pi_2, f_{i-1}]^{MN}+[\pi_1, f_i]^{MN}}_{i},\dots,[\pi_2, f_n]^{MN})\\
&=&-([\pi_1,[\pi_1,f_1]^{MN}]^{MN},[\pi_2,[\pi_1,f_1]^{MN}]^{MN}+[\pi_1,[\pi_2,f_1]^{MN}]^{MN}+[\pi_1,[\pi_1,f_2]^{MN}]^{MN},\dots,\\
&& \underbrace{[\pi_2,[\pi_2,f_{i-2}]^{MN}]^{MN}+[\pi_2,[\pi_1,f_{i-1}]^{MN}]^{MN}+[\pi_1,[\pi_2,f_{i-1}]^{MN}]^{MN}+[\pi_1,[\pi_1,f_i]^{MN}]^{MN}}_{3\leq i\leq n-1},\dots,\\
&&[\pi_2,[\pi_2,f_{n-1}]^{MN}]^{MN}+[\pi_2,[\pi_1,f_n]^{MN}]^{MN}+[\pi_1,[\pi_2,f_n]^{MN}]^{MN},[\pi_2,[\pi_2,f_n]^{MN}]^{MN})\\
&=&-(\frac{1}{2}[[\pi_1,\pi_1]^{MN},f_1]^{MN},[[\pi_1,\pi_2]^{MN},f_1]^{MN}+\frac{1}{2}[[\pi_1,\pi_1]^{MN},f_2]^{MN},\dots,\\
&& \underbrace{\frac{1}{2}[[\pi_2,\pi_2]^{MN},f_{i-2}]^{MN}+[[\pi_1,\pi_2]^{MN},f_{i-1}]^{MN}+\frac{1}{2}[[\pi_1,\pi_1]^{MN},f_i]^{MN}}_{3\leq i\leq n-1},\dots,\\
&&\frac{1}{2}[[\pi_2,\pi_2]^{MN},f_{n-1}]^{MN}+[[\pi_1,\pi_2]^{MN},f_n]^{MN},\frac{1}{2}[[\pi_2,\pi_2]^{MN},f_n]^{MN})\\
&=&(0,0,\dots,0).
\end{eqnarray*}
Thus, we have $\delta\circ\delta=0$.
\end{proof}
\begin{defi}
Let $(\g,\cdot,\ast)$ be a compatible pre-Lie algebra. The cohomology of the cochain complex $(\oplus_{n=1}^{+\infty}\mathfrak{C}^n(\g;\g),\delta)$ is called the cohomology of $(\g,\cdot,\ast)$. The corresponding $n$-th cohomology group is denoted by $\huaH^n(\g;\g)$.
\end{defi}

\subsection{Infinitesimal deformations of compatible pre-Lie algebras}
\begin{defi}
Let $(\g,\cdot,\ast)$ be a compatible pre-Lie algebra and $(\omega_1,\omega_2)\in \mathfrak{C}^2(\g,\g)$. Define
\begin{equation*}
x\cdot_t y=x\cdot y+t\omega_1(x,y),\quad x\ast_t y=x\ast y+t\omega_2(x,y),\quad \forall x,y\in \g.
\end{equation*}
If for all $t\in \mathbb{K}$, $(\g,\cdot_t,\ast_t)$ is still a compatible pre-Lie algebra, then we say that $(\omega_1,\omega_2)$ generates an infinitesimal deformation of $(\g,\cdot,\ast)$.
\end{defi}
It is straightforward to verify that $(\omega_1,\omega_2)$ generates an infinitesimal deformation of a compatible pre-Lie algebra $(\g,\cdot,\ast)$ if and only if for all $k_1,k_2\in \mathbb{K}$, $k_1\omega_1+k_2\omega_2$ generates an infinitesimal deformation of the pre-Lie algebra $(\g,\diamond)$, where $``\diamond$'' is given by \eqref{compatible-product}.

By Theorem \ref{thm-MC}, $(\omega_1,\omega_2)$ generates an infinitesimal deformation of a compatible pre-Lie algebra $(\g,\pi_1,\pi_2)$, where $x\cdot y=\pi_1(x,y)$ and $x\ast y=\pi_2(x,y)$ if and only if
\begin{align}
\label{deformation-1}&[\pi_1,\omega_1]^{MN}=0,\quad [\pi_2,\omega_2]^{MN}=0, \quad [\pi_1,\omega_2]^{MN}+[\pi_2,\omega_1]^{MN}=0,\\
\label{deformation-2}&[\omega_1,\omega_1]^{MN}=0,\quad [\omega_2,\omega_2]^{MN}=0,\quad [\omega_1,\omega_2]^{MN}=0.
\end{align}
Obviously, \eqref{deformation-1} means that $(\omega_1,\omega_2)$ is a $2$-cocycle of the compatible pre-Lie algebra $(\g,\cdot,\ast)$, i.e. $\delta(\omega_1,\omega_2)=0$. \eqref{deformation-2} means that $(\g,\omega_1,\omega_2)$ is a compatible pre-Lie algebra.
\begin{thm}
With the above notation, $(\omega_1,\omega_2)$ is a $2$-cocycle of the compatible pre-Lie algebra $(\g,\cdot,\ast)$.
\end{thm}
\begin{defi}
Two infinitesimal deformations $(\g,\cdot_t,\ast_t)$ and $(\g',\cdot'_t,\ast'_t)$ generated by $(\omega_1,\omega_2)$ and $(\omega'_1,\omega'_2)$ respectively are said to be {\bf equivalent} if there exists a linear operator $N\in\gl(\g)$ such that $\Id+tN$ is a compatible pre-Lie algebra homomorphism from $(\g',\cdot'_t,\ast'_t)$ to $(\g,\cdot_t,\ast_t)$.
\end{defi}
Two infinitesimal deformations $(\g,\cdot_t,\ast_t)$ and $(\g',\cdot'_t,\ast'_t)$ generated by $(\omega_1,\omega_2)$ and $(\omega'_1,\omega'_2)$ respectively are equivalent if and only if for all $x,y\in \g$, the following equalities hold:
\begin{eqnarray}
\label{equi-deformation-1}\omega'_1(x,y)-\omega_1(x,y)&=&N(x)\cdot y+x\cdot N(y)-N(x\cdot y),\\
\label{equi-deformation-2}N(\omega'_1(x,y))&=&\omega_1(x,N(y))+\omega_1(N(x),y)+N(x)\cdot N(y),\\
\label{equi-deformation-3}\omega_1(N(x),N(y))&=&0,\\
\label{equi-deformation-4}\omega'_2(x,y)-\omega_2(x,y)&=&N(x)\ast y+x\ast N(y)-N(x\ast y),\\
\label{equi-deformation-5}N(\omega'_2(x,y))&=&\omega_2(x,N(y))+\omega_2(N(x),y)+N(x)\ast N(y),\\
\label{equi-deformation-6}\omega_2(N(x),N(y))&=&0.
\end{eqnarray}
Note that \eqref{equi-deformation-1} and \eqref{equi-deformation-4} means that $(\omega'_1-\omega_1,\omega'_2-\omega_2)=\delta N=([\pi_1,N]^{MN},[\pi_2,N]^{MN})$. Thus, we have
\begin{thm}
Let $(\g,\cdot,\ast)$ be a compatible pre-Lie algebra. If two infinitesimal deformations $(\g,\cdot_t,\ast_t)$ and $(\g',\cdot'_t,\ast'_t)$ generated by $(\omega_1,\omega_2)$ and $(\omega'_1,\omega'_2)$ respectively are equivalent, then $(\omega_1,\omega_2)$ and $(\omega'_1,\omega'_2)$ are in the same cohomology class of $\huaH^2(\g;\g).$
\end{thm}

\begin{defi}{\rm(\cite{QW})}
Let $(\g,\cdot)$ be a pre-Lie algebra. A linear operator $N\in \gl(\g)$ is called a {\bf Nijenhuis operator} on $(\g,\cdot)$ if for all $x,y\in \g$
      \begin{equation*}
        N(x)\cdot N(y)=N(x\cdot_N y),
        \end{equation*}
        where the product $``\cdot_N$'' is defined by
\begin{equation*}
x\cdot_N y\triangleq  N(x)\cdot y+x\cdot N(y)-N(x\cdot y).
\end{equation*}
\end{defi}

\begin{pro}{\rm(\cite{QW})}\label{deform-Nij}
Let $N$ be a Nijenhuis operator on a pre-Lie algebra $(\g,\cdot)$, then $(\g,\cdot_N)$ is a pre-Lie algebra, and $N$ is a homomorphism from $(\g,\cdot_N)$ to $(\g,\cdot)$.
\end{pro}

\begin{defi}
Let $(\g,\cdot,\ast)$ be a compatible pre-Lie algebra. A linear operator $N\in \gl(\g)$ is called a {\bf Nijenhuis operator} on $(\g,\cdot,\ast)$ if $N$ is both a Nijenhuis operator on the pre-Lie algebra $(\g,\cdot)$ and a Nijenhuis operator on the pre-Lie algebra $(\g,\ast)$.
\end{defi}

\begin{pro}\label{compatible-Nijenhuis}
Let $(\g,\cdot,\ast)$ be a compatible pre-Lie algebra and $N\in \gl(\g)$ a linear map. Then $N$ is a Nijenhuis operator on the compatible pre-Lie algebra $(\g,\cdot,\ast)$ if and only if $N$ is a Nijenhuis operator on the pre-Lie algebra $(\g,\diamond)$, where $``\diamond$'' is given by \eqref{compatible-product}.
\end{pro}
\begin{proof}
If $N$ is a Nijenhuis operator on the compatible pre-Lie algebra $(\g,\cdot,\ast)$, for all $x,y\in \g$, we have
\begin{eqnarray*}
&&N(x)\diamond N(y)-N(N(x)\diamond y+x\diamond N(y)-N(x\diamond y))\\
&=&k_1N(x)\cdot N(y)+k_2N(x)\ast N(y)-N(k_1N(x)\cdot y+k_2N(x)\ast y\\
&&+k_1x\cdot N(y)+k_2x\ast N(y)-N(k_1x\cdot y+k_2x\ast y))\\
&=&k_1(N(x)\cdot N(y)-N(N(x)\cdot y+x\cdot N(y)-N(x\cdot y)))\\
&&+k_2(N(x)\ast N(y)-N(N(x)\ast y+x\ast N(y)-N(x\ast y)))\\
&=&0,
\end{eqnarray*}
which implies that $N$ is a Nijenhuis operator on the pre-Lie algebra $(\g,\diamond)$. The converse part can be proved similarly. We omit details. The proof is finished.
\end{proof}
\begin{pro}\label{deformed-compatible}
Let $N\in \gl(\g)$ be a Nijenhuis operator on the compatible pre-Lie algebra $(\g,\cdot,\ast)$. Then $(\g,\cdot_N,\ast_N)$ is a compatible pre-Lie algebra and $N$ is a homomorphism from $(\g,\cdot_N,\ast_N)$ to $(\g,\cdot,\ast)$.
\end{pro}
\begin{proof}
By Proposition \ref{compatible-Nijenhuis}, $N$ is a Nijenhuis operator on the pre-Lie algebra $(\g,\diamond)$. For all $x,y\in \g$, we have
\begin{eqnarray*}
 x\diamond_N y&=&N(x)\diamond y+x\diamond N(y)-N(x\diamond y)\\
&=&k_1N(x)\cdot y+k_2N(x)\ast y+k_1x\cdot N(y)+k_2x\ast N(y)-k_1N(x\cdot y)-k_2N(x\ast y)\\
&=&k_1(x\cdot_N y)+k_2(x\ast_N y).
\end{eqnarray*}
By Proposition \ref{deform-Nij}, $(\g,\diamond_N)$ is a pre-Lie algebra. By Proposition \ref{linear}, $(\g,\cdot_N,\ast_N)$ is a compatible pre-Lie algebra. By Proposition \ref{deform-Nij}, $N$ is a homomorphism from $(\g,\cdot_N)$ to $(\g,\cdot)$ and a homomorphism from $(\g,\ast_N)$ to $(\g,\ast)$. Thus, $N$ is a homomorphism from $(\g,\cdot_N,\ast_N)$ to $(\g,\cdot,\ast)$.
\end{proof}

\begin{defi}
An infinitesimal deformation $(\g,\cdot_t,\ast_t)$ of a compatible pre-Lie algebra $(\g,\cdot,\ast)$ generated by $(\omega_1,\omega_2)$ is {\bf trivial} if there exists a linear operator $N\in\gl(\g)$ such that $\Id+tN$ is a compatible pre-Lie algebra homomorphism from $(\g,\cdot_t,\ast_t)$ to $(\g,\cdot,\ast)$.
\end{defi}
 $(\g,\cdot_t,\ast_t)$ is a trivial infinitesimal deformation if and only if for all $x,y\in \g$, the following equalities hold:
\begin{eqnarray}
\label{trivial-deformation-1}\omega_1(x,y)&=&N(x)\cdot y+x\cdot N(y)-N(x\cdot y),\\
\label{trivial-deformation-2}N(\omega_1(x,y))&=&N(x)\cdot N(y),\\
\label{trivial-deformation-3}\omega_2(x,y)&=&N(x)\ast y+x\ast N(y)-N(x\ast y),\\
\label{trivial-deformation-4}N(\omega_2(x,y))&=&N(x)\ast N(y).
\end{eqnarray}
By \eqref{trivial-deformation-1} and \eqref{trivial-deformation-2}, we obtain that $N$ is a Nijenhuis operator on the pre-Lie algebra $(\g,\cdot)$. By \eqref{trivial-deformation-3} and \eqref{trivial-deformation-4}, we obtain that $N$ is a Nijenhuis operator on the pre-Lie algebra $(\g,\ast)$.  Thus, by \eqref{trivial-deformation-1}, \eqref{trivial-deformation-2}, \eqref{trivial-deformation-3} and \eqref{trivial-deformation-4}, we obtain that a trivial infinitesimal deformation of a compatible pre-Lie algebra gives rise to a Nijenhuis operator $N$ on the compatible pre-Lie algebra. Conversely, a Nijenhuis operator can also generate a trivial infinitesimal deformation as the following theorem shows.
\begin{thm}
Let N be a Nijenhuis operator on a compatible pre-Lie algebra $(\g,\cdot,\ast)$. Then a infinitesimal deformation $(\g,\cdot_t,\ast_t)$ of the compatible pre-Lie algebra $(\g,\cdot,\ast)$ can be obtained by putting
\begin{eqnarray}
\label{trivial-deformation-generate-1}\omega_1(x,y)&=&N(x)\cdot y+x\cdot N(y)-N(x\cdot y),\\
\label{trivial-deformation-generate-2}\omega_2(x,y)&=&N(x)\ast y+x\ast N(y)-N(x\ast y).
\end{eqnarray}
Furthermore, this infinitesimal deformation $(\g,\cdot_t,\ast_t)$ is trivial.
\end{thm}
\begin{proof}
By \eqref{trivial-deformation-generate-1} and \eqref{trivial-deformation-generate-2}, we obtain that $(\omega_1,\omega_2)=\delta N$. Thus, $(\omega_1,\omega_2)$ is a $2$-cocycle of the compatible pre-Lie algebra $(\g,\cdot,\ast)$. Since N is a Nijenhuis operator on a compatible pre-Lie algebra $(\g,\cdot,\ast)$, by Proposition \ref{deformed-compatible}, $(\g,\omega_1,\omega_2)$ is a compatible pre-Lie algebra. Thus, $(\g,\cdot_t,\ast_t)$ is a infinitesimal deformation of $(\g,\cdot,\ast)$. It is straightforward to deduce that $\Id+tN$ is a compatible pre-Lie algebra homomorphism from $(\g,\cdot_t,\ast_t)$ to $(\g,\cdot,\ast)$. Thus, this infinitesimal deformation is trivial.
\end{proof}
\subsection{Formal deformations of compatible pre-Lie algebras}
\begin{defi}
Let $(\g,\pi_1,\pi_2)$ be a compatible pre-Lie algebra, $\pi_1^t=\pi_1+\sum_{i=1}^{+\infty} \pi_1^it^i,\pi_2^t=\pi_2+\sum_{i=1}^{+\infty} \pi_2^it^i:\g[[t]]\otimes \g[[t]]\longrightarrow \g[[t]]$ be $\mathbb K[[t]]$-bilinear maps, where $\pi_1^i,\pi_2^i:\g\otimes \g\longrightarrow \g$ are linear maps. If $(\g[[t]],\pi_1^t,\pi_2^t)$ is still a compatible pre-Lie algebra, we say that $\{\pi_1^i,\pi_2^i\}_{i\geq1}$ generates a {\bf $1$-parameter formal deformation} of a compatible pre-Lie algebra $(\g,\pi_1,\pi_2)$.
\end{defi}

If $\{\pi_1^i,\pi_2^i\}_{i\geq1}$ generates a $1$-parameter formal deformation of a compatible pre-Lie algebra $(\g,\pi_1,\pi_2)$, for all $x,y,z\in \g$ and $n=1,2,\dots$,  we have
\begin{equation*}
\sum_{i+j=n\atop i,j\geq 0}\pi_1^i(\pi_1^j(x,y),z)-\pi_1^i(x,\pi_1^j(y,z))-\pi_1^i(\pi_1^j(y,x),z)+\pi_1^i(y,\pi_1^j(x,z))=0.
\end{equation*}
Moreover, we have
\begin{eqnarray}\label{n-sum-cocycle-1}
&&\sum_{i+j=n\atop 0<i,j\leq n-1 }\pi_1^i(\pi_1^j(x,y),z)-\pi_1^i(x,\pi_1^j(y,z))-\pi_1^i(\pi_1^j(y,x),z)+\pi_1^i(y,\pi_1^j(x,z))=-[\pi_1,\pi_1^n]^{MN}(x,y,z).
\end{eqnarray}
Similarly, we have
\begin{eqnarray}\label{n-sum-cocycle-2}
&&\sum_{i+j=n\atop 0<i,j\leq n-1 }\pi_2^i(\pi_2^j(x,y),z)-\pi_2^i(x,\pi_2^j(y,z))-\pi_2^i(\pi_2^j(y,x),z)+\pi_2^i(y,\pi_2^j(x,z))=-[\pi_2,\pi_2^n]^{MN}(x,y,z).
\end{eqnarray}
For all $x,y,z\in \g$ and $n=1,2,\dots$,  we have
\begin{eqnarray*}
&&\sum_{i+j=n\atop i,j\geq 0}\pi_1^i(\pi_2^j(x,y),z)+\pi_2^j(\pi_1^i(x,y),z)-\pi_1^i(x,\pi_2^j(y,z))-\pi_2^j(x,\pi_1^i(y,z))\\
&&-\pi_1^i(\pi_2^j(y,x),z)-\pi_2^j(\pi_1^i(y,x),z)+\pi_1^i(y,\pi_2^j(x,z))+\pi_2^j(y,\pi_1^i(x,z))=0.
\end{eqnarray*}
Moreover, we have
\begin{eqnarray*}\label{n-sum-cocycle-3}
&&\sum_{i+j=n\atop 0<i,j\leq n-1 }\pi_1^i(\pi_2^j(x,y),z)+\pi_2^j(\pi_1^i(x,y),z)-\pi_1^i(x,\pi_2^j(y,z))-\pi_2^j(x,\pi_1^i(y,z))\\
&&-\pi_1^i(\pi_2^j(y,x),z)-\pi_2^j(\pi_1^i(y,x),z)+\pi_1^i(y,\pi_2^j(x,z))+\pi_2^j(y,\pi_1^i(x,z))\\
&=&-([\pi_1,\pi_2^n]^{MN}+[\pi_2,\pi_1^n]^{MN})(x,y,z).
\end{eqnarray*}

\begin{defi}
Let $({\pi_1^t}'=\pi_1+\sum_{i=1}^{+\infty} {\pi_1^i}'t^i,{\pi_2^t}'=\pi_2+\sum_{i=1}^{+\infty} {\pi_2^i}'t^i)$ and $(\pi_1^t=\pi_1+\sum_{i=1}^{+\infty} \pi_1^it^i,\pi_2^t=\pi_2+\sum_{i=1}^{+\infty} \pi_2^it^i)$ be two $1$-parameter formal deformations of a compatible pre-Lie algebra $(\g,\pi_1,\pi_2)$. A {\bf formal isomorphism} from $(\g[[t]],{\pi_1^t}',{\pi_2^t}')$ to $(\g[[t]],\pi_1^t,\pi_2^t)$ is a power series $\Phi_t=\sum_{i=0}^{+\infty} \varphi_it^i$, where $\varphi_i:A\longrightarrow A$ are linear maps with $\varphi_0={\Id}$, such that
\begin{eqnarray*}
 \Phi_t\circ {\pi_1^t}'&=&\pi_1^t \circ(\Phi_t\times \Phi_t),\\
   \Phi_t\circ {\pi_2^t}'&=&\pi_2^t \circ(\Phi_t\times \Phi_t).
\end{eqnarray*}
Two $1$-parameter formal deformations $(\g[[t]],{\pi_1^t}',{\pi_2^t}')$ and $(\g[[t]],\pi_1,\pi_2)$ are said to be {\bf equivalent} if there exists a
formal isomorphism $\Phi_t=\sum_{i=0}^{+\infty} \varphi_it^i$ from $(\g[[t]],{\pi_1^t}',{\pi_2^t}')$ to $(\g[[t]],\pi_1,\pi_2)$.
\end{defi}

\begin{defi}
A $1$-parameter formal deformation $(\g[[t]],\pi_1^t,\pi_2^t)$ of a compatible pre-Lie algebra $(\g,\pi_1,\pi_2)$ is said to be {\bf trivial} if it is equivalent to $(\g,\pi_1,\pi_2)$, i.e. there exists $\Phi_t=\sum_{i=0}^{+\infty} \varphi_it^i$, where $\varphi_i:A\longrightarrow A$ are linear maps with $\varphi_0={\Id}$, such that
\begin{eqnarray*}
 \Phi_t\circ \pi_1^t&=&\pi_1 \circ(\Phi_t\times \Phi_t),\\
 \Phi_t\circ \pi_2^t&=&\pi_2 \circ(\Phi_t\times \Phi_t).
\end{eqnarray*}
\end{defi}

\begin{defi}
Let $(\g,\pi_1,\pi_2)$ be a compatible pre-Lie algebra. If all $1$-parameter formal deformations are trivial, we say that $(\g,\pi_1,\pi_2)$ is {\bf rigid}.
\end{defi}

\begin{thm}
Let $(\g,\pi_1,\pi_2)$ be a compatible pre-Lie algebra. If $\huaH^2(\g;\g)=0$, then $(\g,\pi_1,\pi_2)$ is rigid.
\end{thm}
\begin{proof}
Let $(\pi_1^t=\pi_1+\sum_{i=1}^{+\infty} \pi_1^it^i,\pi_2^t=\pi_2+\sum_{i=1}^{+\infty} \pi_2^it^i)$ be a $1$-parameter formal deformation and assume that $n\geq1$ is the minimal number such that $(\pi_1^n,\pi_2^n)$ is not zero. By \eqref{n-sum-cocycle-1}, \eqref{n-sum-cocycle-2}, \eqref{n-sum-cocycle-3} and $\huaH^2(\g;\g)=0$, we have $(\pi_1^n,\pi_2^n)\in B^2(A;A)$. Thus, there exists $\varphi_n \in \mathfrak{C}^1(\g;\g)$ such that $(\pi_1^n,\pi_2^n)=\delta(-\varphi_n)=(\delta_{\pi_1}(-\varphi_n),\delta_{\pi_2}(-\varphi_n))$. Let $\Phi_t={\Id}+\varphi_nt^n$ and define a new formal deformation $({\pi_1^t}',{\pi_2^t}')$ by ${\pi_1^t}'(x,y)=\Phi_t^{-1}\circ\pi_1^t(\Phi_t(x),\Phi_t(y))$, ${\pi_2^t}'(x,y)=\Phi_t^{-1}\circ\pi_2^t(\Phi_t(x),\Phi_t(y))$. Then $({\pi_1^t}',{\pi_2^t}')$ and $(\pi_1^t,\pi_2^t)$ are equivalent. By straightforward computation, for all $x,y\in \g$, we have
\begin{eqnarray*}
{\pi_1^t}'(x,y)
&=&\Phi_t^{-1}\circ\pi_1^t(\Phi_t(x),\Phi_t(y))\\
&=&({\Id}-\varphi_n t^n+\dots)\pi_1^t\big(x+\varphi_n(x)t^n,y+\varphi_n(y)t^n\big)\\
&=&({\Id}-\varphi_n t^n+\dots)\Big(x\cdot y+\big(x\cdot\varphi_n(y)+\varphi_n(x)\cdot y+\pi_1^n(x,y)\big)t^n+\dots\Big)\\
&=&x\cdot y+\Big(x\cdot\varphi_n(y)+\varphi_n(x)\cdot y+\pi_1^n(x,y)-\varphi_n(x\cdot y)\Big)t^n+\dots.
\end{eqnarray*}
Thus, we have ${\pi_1^1}'={\pi_1^2}'=\dots={\pi_1^{n-1}}'=0$. Moreover, we have
\begin{eqnarray*}
{\pi_1^n}'(x,y)&=&x\cdot \varphi_n(y)+\varphi_n(x)\cdot y+\pi_1^n(x,y)-\varphi_n(x\cdot y)\\
&=&\delta_{\pi_1}\varphi_n(x,y)+\pi_1^n(x,y)\\
&=&0.
\end{eqnarray*}
Similarly, we have ${\pi_2^n}'(x,y)=0$.
Keep repeating the process, we obtain that $(\g[[t]],\pi_1^t,\pi_2^t)$ is equivalent to $(\g,\pi_1,\pi_2)$. The proof is finished.
\end{proof}

\section{Abelian extensions of compatible pre-Lie algebras}\label{sec:extension}
In this section, first, we give a compatible pre-Lie algebra $(\g,\pi_1,\pi_2)$ and its representation $(V,\rho,\mu,\tilde{\rho},\tilde{\mu})$. We construct a bidifferential graded Lie algebra whose Maurer-Cartan elements is $(\pi_1+\rho+\mu,\pi_2+\tilde{\rho}+\tilde{\mu})$. Then, we give a cohomology of a compatible pre-Lie algebra with coefficients in arbitrary representation. Finally, we study abelian extensions of compatible pre-Lie algebras using this cohomological approach. We show that abelian extensions are classified by the second cohomology group.
\subsection{Cohomologies of compatible pre-Lie algebras with coefficients in arbitrary representation}
\
\newline
\indent
Let $\g_1$ and $\g_2$ be vector spaces and elements in $\g_1$ will be denoted by $x,y,x_i$ and elements in $\g_2$ will be denoted by $u,v,v_i$. Let $c:\wedge^{n-1}\g_1\otimes \g_1\longrightarrow \g_2$ be a linear map. We can construct a linear map $\hat{c}\in C^n(\g_1\oplus \g_2,\g_1\oplus \g_2)$ by
\begin{equation*}
\hat{c}(x_1+v_1,\dots,x_n+v_n):=c(x_1,\dots,x_n).
\end{equation*}
In general, for a given linear map $f:\wedge^{k-1}\g_1\otimes\wedge^l\g_2\otimes\g_1\longrightarrow\g_j$ for $j\in \{1,2\}$, we define a linear map $\hat{f}\in C^{k+l}(\g_1\oplus \g_2,\g_1\oplus \g_2)$ by
\begin{equation*}
\hat{f}(x_1+v_1,\dots,x_{k+l}+v_{k+l})=\sum_{\sigma\in \mathbb{S}(k-1,l)}\sgn(\sigma)f(x_{\sigma(1)},\dots,x_{\sigma(k-l)},v_{\sigma(k)},\dots,v_{\sigma(k+l-1)},x_{k+l}).
\end{equation*}
Similarly, for $f:\wedge^k\g_1\otimes\wedge^{l-1}\g_2\otimes\g_2\longrightarrow\g_j$ for $j\in \{1,2\}$, we define a linear map $\hat{f}\in C^{k+l}(\g_1\oplus \g_2,\g_1\oplus \g_2)$ by
\begin{equation*}
\hat{f}(x_1+v_1,\dots,x_{k+l}+v_{k+l})=\sum_{\sigma\in \mathbb{S}(k,l-1)}\sgn(\sigma)f(x_{\sigma(1)},\dots,x_{\sigma(k)},v_{\sigma(k+1)},\dots,v_{\sigma(k+l-1)},v_{k+l}).
\end{equation*}
We call the linear map $\hat{f}$ a {\bf horizontal lift} of $f$, or simply a lift.
We define $\huaG^{k,l}=\wedge^{k-1}\g_1\otimes\wedge^l\g_2\otimes\g_1+\wedge^k\g_1\otimes\wedge^{l-1}\g_2\otimes\g_2$. The vector space $\wedge^{n-1}(\g_1\oplus\g_2)\otimes(\g_1\oplus\g_2)$ is isomorphic to the direct sum of  $\huaG^{k,l}, k+l=n$. In the sequel, we will omit the notation $\hat{\cdot}$.
\begin{defi}{\rm(\cite{Liu})}
A linear map $f\in \Hom(\wedge^{n-1}(\g_1\oplus\g_2)\otimes(\g_1\oplus\g_2),\g_1\oplus\g_2)$ has a {\bf bidegree} $k|l$ if the following four conditions hold:
\begin{itemize}
\item [$\rm(i)$]  $k+l+1=n$;
\item[$\rm(ii)$] If X is an element in $\huaG^{k+1,l}$, then $f(X)\in \g_1$;
\item[$\rm(iii)$] If X is an element in $\huaG^{k,l+1}$, then $f(X)\in \g_2$;
\item[$\rm(iv)$] All the other case, $f(X)=0$.
\end{itemize}
We denote a linear map $f$ with bidegree $k|l$ by $||f||=k|l$.
\end{defi}
We call a linear map $f$ {\bf homogeneous} if $f$ has a bidegree. We denote the set of homogeneous linear maps of bidegree $k|l$ by $C^{k|l}(\g_1\oplus \g_2,\g_1\oplus \g_2)$. We have $k+l\geq 0,k,l\geq -1$  because $n\geq 1$ and $k+1,l+1\geq 0$.

By the above lift, we have the following isomorphisms:
\begin{eqnarray*}
C^{k|0}(\g_1\oplus \g_2,\g_1\oplus \g_2)&\cong& \Hom(\wedge^k\g_1\otimes\g_1,\g_1)\oplus\Hom(\wedge^k\g_1\otimes\g_2,\g_2)\oplus\Hom(\wedge^{k-1}\g_1\otimes\g_2\otimes\g_1,\g_2);\\
C^{l|-1}(\g_1\oplus \g_2,\g_1\oplus \g_2)&\cong& \Hom(\wedge^{l-1}\g_1\otimes\g_1,\g_2).
\end{eqnarray*}
\begin{lem}{\rm(\cite{Liu})}\label{bidegree-1}
If $||f||=k_f|l_f$ and $||g||=k_g|l_g$, then $[f,g]^{MN}$ has the bidegree $k_f+k_g|l_f+l_g$.
\end{lem}

\begin{pro}{\rm(\cite{Liu1})}\label{MC-prelierep}
Let $(V,\rho,\mu)$ be a representation of the pre-Lie algebra $(\g,\pi)$. Then we have
\begin{equation*}
[\pi+\rho+\mu,\pi+\rho+\mu]^{MN}=0.
\end{equation*}
\end{pro}

Let $(V,\rho,\mu)$ be a representation of the pre-Lie algebra $(\g,\pi)$, where $\pi(x,y)=x\cdot y$. Denote the set of $n$-cochains by
\begin{equation*}
C^n(\g;V)=\Hom(\wedge^{n-1}\g\otimes \g,V),\quad n\geq 1.
\end{equation*}

By Proposition \ref{MC-prelierep} and graded Jacobi identity, we define a coboundary operator $\partial_{\pi+\rho+\mu}:C^n(\g,V)\longrightarrow C^{n+1}(\g,V)$ by
\begin{equation*}
\partial_{\pi+\rho+\mu} f=(-1)^{n-1}[\pi+\rho+\mu,f]^{MN},\quad \forall f\in C^n(\g,V).
\end{equation*}
In facet, since $\pi+\rho+\mu\in C^{1|0}(\g\oplus V,\g\oplus V)$ and $f\in C^{n|-1}(\g\oplus V,\g\oplus V)$, by Lemma \ref{bidegree-1}, we obtain that $[\pi+\rho+\mu,f]^{MN}\in C^{n+1|-1}(\g\oplus V,\g\oplus V)$. Thus, $[\pi+\rho+\mu,f]^{MN}\in C^{n+1}(\g,V)$, we obtain a well-defined cochain complex $(\oplus_{n=1}^{+\infty}C^n(\g;V),\partial_{\pi+\rho+\mu})$. More precisely, for all $x_1,\dots,x_{n+1}\in \g$, we have
\begin{eqnarray*}
(\partial_{\pi+\rho+\mu} f)(x_1,\dots,x_{n+1})&=&\sum_{i=1}^n(-1)^{i+1}\rho(x_i)f(x_1,\dots,\hat{x_i},\dots,x_{n+1})\\
&&+\sum_{i=1}^n(-1)^{i+1}\mu(x_{n+1})f(x_1,\dots,\hat{x_i},\dots,x_n,x_i)\\
&&-\sum_{i=1}^n(-1)^{i+1}f(x_1,\dots,\hat{x_i}\dots,x_n,x_i\cdot x_{n+1})\\
&&+\sum_{1\leq i<j\leq n}(-1)^{i+j} f([x_i,x_j]_C,x_1,\dots,\hat{x_i},\dots,\hat{x_j},\dots,x_{n+1}),
\end{eqnarray*}
which is a coboundary operator of pre-Lie algebra $(\g,\pi)$ with coefficients in the representation $(V,\rho,\mu)$. We can see more details in \cite{DA}.

Let $(\g,\cdot,\ast)$ be a compatible pre-Lie algebra with $\pi_1(x,y)=x\cdot y$ , $\pi_2(x,y)=x\ast y$ and $(V,\rho,\mu,\tilde{\rho},\tilde{\mu})$ a representation of $(\g,\cdot,\ast)$.
\begin{pro}\label{pro:MC}
With the above notation, $(\pi_1+\rho+\mu,\pi_2+\tilde{\rho}+\tilde{\mu})$ is a Maurer-Cartan element of the bidifferential graded Lie algebra $(C^*(\g\oplus V;\g\oplus V),[\cdot,\cdot]^{MN},\dM_1=0,\dM_2=0)$, i.e.
\begin{equation*}
[\pi_1+\rho+\mu,\pi_1+\rho+\mu]^{MN}=0,\quad [\pi_2+\tilde{\rho}+\tilde{\mu},\pi_2+\tilde{\rho}+\tilde{\mu}]^{MN}=0, \quad [\pi_1+\rho+\mu,\pi_2+\tilde{\rho}+\tilde{\mu}]^{MN}=0.
\end{equation*}
\end{pro}
\begin{proof}
Since $(V,\rho,\mu)$ is a representation of the pre-Lie algebra $(\g,\cdot)$, by Proposition \ref{MC-prelierep}, we have
\begin{equation*}
[\pi_1+\rho+\mu,\pi_1+\rho+\mu]^{MN}=0.
\end{equation*}
Similarly, since $(V,\tilde{\rho},\tilde{\mu})$ is a representation of the pre-Lie algebra $(\g,\ast)$, by Proposition \ref{MC-prelierep}, we have
\begin{equation*}
[\pi_2+\tilde{\rho}+\tilde{\mu},\pi_2+\tilde{\rho}+\tilde{\mu}]^{MN}=0.
\end{equation*}
For all $x,y,z\in \g, u,v,w\in V$, by \eqref{compatible-pre-Lie}, \eqref{rep-1} and \eqref{rep-2}, we have
\begin{eqnarray*}
&&[\pi_1+\rho+\mu,\pi_2+\tilde{\rho}+\tilde{\mu}]^{MN}(x+u,y+v,z+w)\\
&=&(\pi_1+\rho+\mu)(\pi_2(x,y)+\tilde{\rho}(x)v+\tilde{\mu}(y)u,z+w)-(\pi_1+\rho+\mu)(\pi_2(y,x)+\tilde{\rho}(y)u+\tilde{\mu}(x)v,z+w)\\
&&-(\pi_1+\rho+\mu)(x+u,\pi_2(y,z)+\tilde{\rho}(y)w+\tilde{\mu}(z)v)+(\pi_1+\rho+\mu)(y+v,\pi_2(x,z)+\tilde{\rho}(x)w+\tilde{\mu}(z)u)\\
&&+(\pi_2+\tilde{\rho}+\tilde{\mu})(\pi_1(x,y)+\rho(x)v+\mu(y)u,z+w)-(\pi_2+\tilde{\rho}+\tilde{\mu})(\pi_1(y,x)+\rho(y)u+\mu(x)v,z+w)\\
&&-(\pi_2+\tilde{\rho}+\tilde{\mu})(x+u,\pi_1(y,z)+\rho(y)w+\mu(z)v)+(\pi_2+\tilde{\rho}+\tilde{\mu})(y+v,\pi_1(x,z)+\rho(x)w+\mu(z)u)\\
&=&\pi_1(\pi_2(x,y),z)+\rho(\pi_2(x,y))w+\mu(z)\tilde{\rho}(x)v+\mu(z)\tilde{\mu}(y)u\\
&&-\pi_1(\pi_2(y,x),z)-\rho(\pi_2(y,x))w-\mu(z)\tilde{\rho}(y)u-\mu(z)\tilde{\mu}(x)v\\
&&-\pi_1(x,\pi_2(y,z))-\rho(x)\tilde{\rho}(y)w-\rho(x)\tilde{\mu}(z)v-\mu(\pi_2(y,z))u\\
&&+\pi_1(y,\pi_2(x,z))+\rho(y)\tilde{\rho}(x)w+\rho(y)\tilde{\mu}(z)u+\mu(\pi_2(x,z))v\\
&&+\pi_2(\pi_1(x,y),z)+\tilde{\rho}(\pi_1(x,y))w+\tilde{\mu}(z)\rho(x)v+\tilde{\mu}(z)\mu(y)u\\
&&-\pi_2(\pi_1(y,x),z)-\tilde{\rho}(\pi_1(y,x))w-\tilde{\mu}(z)\rho(y)u-\tilde{\mu}(z)\mu(x)v\\
&&-\pi_2(x,\pi_1(y,z))-\tilde{\rho}(x)\rho(y)w-\tilde{\rho}(x)\mu(z)v-\tilde{\mu}(\pi_1(y,z))u\\
&&+\pi_2(y,\pi_1(x,z))+\tilde{\rho}(y)\rho(x)w+\tilde{\rho}(y)\mu(z)u+\tilde{\mu}(\pi_1(x,z))v\\
&=&0,
\end{eqnarray*}

which implies that
\begin{equation*}
 [\pi_1+\rho+\mu,\pi_2+\tilde{\rho}+\tilde{\mu}]^{MN}=0.
\end{equation*}
This finishes the proof.
\end{proof}
We define the set of $n$-cochains $(n\geq1)$ by
\begin{equation*}
\mathfrak{C}^n(\g;V)=\underbrace{C^n(\g;V)\oplus C^n(\g;V)\oplus\cdots \oplus C^n(\g;V)}_{n~copies}.
\end{equation*}
Define the operator $\partial:\mathfrak{C}^n(\g;V)\longrightarrow \mathfrak{C}^{n+1}(\g;V)$ by
\begin{equation*}
 \partial^1 f=(\partial_{\pi_1+\rho+\mu}f,\partial_{\pi_2+\tilde{\rho}+\tilde{\mu}}f),\quad
 \forall f \in  \Hom(\g,V), n=1.
 \end{equation*}
And for all $(f_1,\dots,f_n) \in \mathfrak{C}^n(\g,V), 2\leq i\leq n$, we have
\begin{equation*}
\partial^n (f_1,\dots,f_n)=(\partial_{\pi_1+\rho+\mu}f_1,\dots, \underbrace{\partial_{\pi_2+\tilde{\rho}+\tilde{\mu}} f_{i-1}+\partial_{\pi_1+\rho+\mu} f_i}_{i},\dots,\partial_{\pi_2+\tilde{\rho}+\tilde{\mu}} f_n).
\end{equation*}
\begin{thm}
The operator  $\partial:\mathfrak{C}^n(\g;V)\longrightarrow \mathfrak{C}^{n+1}(\g;V)$ defined as above satisfies $\partial\circ\partial=0$.
\end{thm}
\begin{proof}
By Proposition \ref{pro:MC} and the graded Jacobi identity, similarly to the proof of Theorem \ref{coboundary-operaor}, we have $\partial\circ\partial=0$.
\end{proof}
\begin{defi}
Let $(V,\rho,\mu,\tilde{\rho},\tilde{\mu})$ be a representation of the compatible pre-Lie algebra $(\g,\cdot,\ast)$. The cohomology of the cochain complex $(\oplus_{n=1}^{+\infty}\mathfrak{C}^n(\g;V),\partial)$ is called the cohomology of $(\g,\cdot,\ast)$ with coefficients in the representation $(V,\rho,\mu,\tilde{\rho},\tilde{\mu})$. The corresponding $n$-th cohomology group is denoted by $\huaH^n(\g;V)$.
\end{defi}

\subsection{Abelian extensions of compatible pre-Lie algebras}
\begin{defi}
Let $(\g,\cdot,\ast)$ and $(V,\cdot_V,\ast_V)$ be two compatible pre-Lie algebras. An {\bf  extension} of $(\g,\cdot,\ast)$ by $(V,\cdot_V,\ast_V)$ is a short exact sequence of compatible pre-Lie algebras morphisms:
$$
0\longrightarrow V\stackrel{\tau}{\longrightarrow} \hat{\g}\stackrel{p}{\longrightarrow} \g\longrightarrow 0,
$$
where $(\hat{\g},\cdot_{\hat{\g}},\ast_{\hat{\g}})$ is a compatible pre-Lie algebra.

It is called an {\bf abelian extension} if  $(V,\cdot_V,\ast_V)$ is an abelian compatible pre-Lie algebra, i.e.   for all $u,v\in V, u\cdot_V v=u\ast_V v=0$ .
\end{defi}
\begin{defi}
A {\bf section} of an extension $(\hat{\g},\cdot_{\hat{\g}},\ast_{\hat{\g}})$ of a compatible pre-Lie algebra $(\g,\cdot,\ast)$ by  $(V,\cdot_V,\ast_V)$ is a linear map $s:\g\longrightarrow \hat{\g}$ such that $p\circ s=\Id_{\g}$.
\end{defi}

Let $(\hat{\g},\cdot_{\hat{\g}},\ast_{\hat{\g}})$ be an abelian extension of a compatible pre-Lie algebra $(\g,\cdot,\ast)$ by  $(V,\cdot_V,\ast_V)$ and $s:\g\longrightarrow \hat{\g}$ a section. For all $x,y\in \g$, define linear maps $\theta,\tilde{\theta}:\g\otimes \g\longrightarrow V$ respectively by
\begin{eqnarray}
\label{product-1}\theta(x,y)&=&s(x)\cdot_{\hat{\g}}s(y)-s(x\cdot y),\\
\label{product-2}\tilde{\theta}(x,y)&=&s(x)\ast_{\hat{\g}}s(y)-s(x\ast y).
\end{eqnarray}
And for all $x,y\in \g, u\in V$, define $\rho,\mu,\tilde{\rho},\tilde{\mu}:\g\longrightarrow\gl(V)$ respectively by
\begin{eqnarray}
\label{representation-1}\rho(x)(u)=s(x)\cdot_{\hat{\g}} u,\quad \mu(x)(u)=u\cdot_{\hat{\g}} s(x),\\
\label{representation-2}\tilde{\rho}(x)(u)=s(x)\ast_{\hat{\g}} u,\quad \tilde{\mu}(x)(u)=u\ast_{\hat{\g}} s(x).
\end{eqnarray}
Obviously, $\hat{\g}$ is isomorphic to $\g\oplus V$ as vector spaces. Transfer the compatible pre-Lie algebra structure on $\hat{\g}$ to that on $\g\oplus V$, we obtain a compatible pre-Lie algebra $(\g\oplus V,\cdot_{(\theta,\rho,\mu)},\ast_{(\tilde{\theta},\tilde{\rho},\tilde{\mu})})$, where $``\cdot_{(\theta,\rho,\mu)}$'' and $``\ast_{(\tilde{\theta},\tilde{\rho},\tilde{\mu})}$'' are given by
\begin{eqnarray}
  \label{new-compatible-1}(x+u)\cdot_{(\theta,\rho,\mu)} (y+v)&=&x\cdot y+\theta(x,y)+\rho(x)(v)+\mu(y)(u),\quad \forall ~x,y\in \g, u,v\in V,\\
   \label{new-compatible-2}(x+u)\ast_{(\tilde{\theta},\tilde{\rho},\tilde{\mu})} (y+v)&=&x\ast y+\tilde{\theta}(x,y)+\tilde{\rho}(x)(v)+\tilde{\mu}(y)(u),\quad \forall ~x,y\in \g, u,v\in V.
\end{eqnarray}

\begin{thm}\label{thm:representation}
With the above notation, $(V,\rho,\mu,\tilde{\rho},\tilde{\mu})$ is a representation of the compatible pre-Lie algebra $(\g,\cdot,\ast)$. Moreover, this representation is independent of the choice of sections.
\end{thm}
\begin{proof}
For all $x,y\in \g$, $u\in V$, by the definition of a pre-Lie algebra, we have
\begin{eqnarray*}0&=&(x\cdot_{(\theta,\rho,\mu)}y)\cdot_{(\theta,\rho,\mu)} u-x\cdot_{(\theta,\rho,\mu)} (y\cdot_{(\theta,\rho,\mu)} u)-(y\cdot_{(\theta,\rho,\mu)} x)\cdot_{(\theta,\rho,\mu)} u+y\cdot_{(\theta,\rho,\mu)} (x\cdot_{(\theta,\rho,\mu)} u)\\
&=&(x\cdot y+\theta(x,y))\cdot_{(\theta,\rho,\mu)} u-x\cdot_{(\theta,\rho,\mu)}\rho(y)u-(y\cdot x+\theta(y,x))\cdot_{(\theta,\rho,\mu)}u+y\cdot_{(\theta,\rho,\mu)}\rho(x)u\\
&=&\rho(x\cdot y)u-\rho(x)\rho(y)u-\rho(y\cdot x)u+\rho(y)\rho(x)u,
\end{eqnarray*}

and
\begin{eqnarray*}
0&=&(u\cdot_{(\theta,\rho,\mu)}x)\cdot_{(\theta,\rho,\mu)}y-u\cdot_{(\theta,\rho,\mu)}(x\cdot_{(\theta,\rho,\mu)} y)-(x\cdot_{(\theta,\rho,\mu)} u)\cdot_{(\theta,\rho,\mu)} y+x\cdot_{(\theta,\rho,\mu)}(u\cdot_{(\theta,\rho,\mu)} y)\\
&=&\mu(x)u\cdot_{(\theta,\rho,\mu)} y-u\cdot_{(\theta,\rho,\mu)}(x\cdot y+\theta(x,y))-\rho(x)u\cdot_{(\theta,\rho,\mu)} y+x\cdot_{(\theta,\rho,\mu)} \mu(y)u\\
&=&\mu(y)\mu(x)u-\mu(x \cdot y)u-\mu(y)\rho(x)u+\rho(x)\mu(y)u,
\end{eqnarray*}
which implies that
\begin{eqnarray*}
\rho([x,y]_C&=&\rho(x)\circ\rho(y)-\rho(y)\circ\rho(x),\\
\mu(y)\circ\mu(x)-\mu(x\cdot y)&=&\mu(y)\circ\rho(x)-\rho(x)\circ\mu(y).
\end{eqnarray*}
Thus, $(V,\rho,\mu)$ is a representation of the pre-Lie algebra $(\g,\cdot)$. Similarly, $(V,\tilde{\rho},\tilde{\mu})$ is a representation of the pre-Lie algebra $(\g,\ast)$.
\begin{eqnarray*}0&=&(x\ast_{(\tilde{\theta},\tilde{\rho},\tilde{\mu})} y)\cdot_{(\theta,\rho,\mu)} u+(x\cdot_{(\theta,\rho,\mu)} y)\ast_{(\tilde{\theta},\tilde{\rho},\tilde{\mu})} u-x\cdot_{(\theta,\rho,\mu)}(y\ast_{(\tilde{\theta},\tilde{\rho},\tilde{\mu})} u)-x\ast_{(\tilde{\theta},\tilde{\rho},\tilde{\mu})}(y\cdot_{(\theta,\rho,\mu)} u)\\
&&-(y\ast_{(\tilde{\theta},\tilde{\rho},\tilde{\mu})} x)\cdot_{(\theta,\rho,\mu)} u-(y\cdot_{(\theta,\rho,\mu)} x)\ast_{(\tilde{\theta},\tilde{\rho},\tilde{\mu})} u+y\cdot_{(\theta,\rho,\mu)}(x\ast_{(\tilde{\theta},\tilde{\rho},\tilde{\mu})} u)+y\ast_{(\tilde{\theta},\tilde{\rho},\tilde{\mu})}(x\cdot_{(\theta,\rho,\mu)} u)\\
&=&(x\ast y+\tilde{\theta}(x,y))\cdot_{(\theta,\rho,\mu)} u+(x\cdot y+\theta(x,y))\ast_{(\tilde{\theta},\tilde{\rho},\tilde{\mu})} u-x\cdot_{(\theta,\rho,\mu)}\tilde{\rho}(y)u-x\ast_{(\tilde{\theta},\tilde{\rho},\tilde{\mu})} \rho(y)u\\
&&-(y\ast x+\tilde{\theta}(y,x))\cdot_{(\theta,\rho,\mu)} u-(y\cdot x+\theta(y,x))\ast_{(\tilde{\theta},\tilde{\rho},\tilde{\mu})} u+y\cdot_{(\theta,\rho,\mu)}\tilde{\rho}(x)u+y\ast_{(\tilde{\theta},\tilde{\rho},\tilde{\mu})} \rho(x)u\\
&=&\rho(x\ast y)u+\tilde{\rho}(x\cdot y)u-\rho(x)\tilde{\rho}(y)u-\tilde{\rho}(x)\rho(y)u-\rho(y\ast x)u-\tilde{\rho}(y\cdot x)u+\rho(y)\tilde{\rho}(x)u+\tilde{\rho}(y)\rho(x)u,
\end{eqnarray*}
which implies that
\begin{equation*}
\rho(x\ast y)+\tilde{\rho}(x\cdot y)-\rho(x)\tilde{\rho}(y)-\tilde{\rho}(x)\rho(y)=\rho(y\ast x)+\tilde{\rho}(y\cdot x)-\rho(y)\tilde{\rho}(x)-\tilde{\rho}(y)\rho(x).
\end{equation*}
Similarly, by
\begin{eqnarray*}
&&(x\ast_{(\tilde{\theta},\tilde{\rho},\tilde{\mu})} u)\cdot_{(\theta,\rho,\mu)} y+(x\cdot_{(\theta,\rho,\mu)} u)\ast_{(\tilde{\theta},\tilde{\rho},\tilde{\mu})} y-x\cdot_{(\theta,\rho,\mu)}(u\ast_{(\tilde{\theta},\tilde{\rho},\tilde{\mu})} y)-x\ast_{(\tilde{\theta},\tilde{\rho},\tilde{\mu})}(u\cdot_{(\theta,\rho,\mu)} y)\\
&&-(u\ast_{(\tilde{\theta},\tilde{\rho},\tilde{\mu})} x)\cdot_{(\theta,\rho,\mu)} y-(u\cdot_{(\theta,\rho,\mu)} x)\ast_{(\tilde{\theta},\tilde{\rho},\tilde{\mu})} y+u\cdot_{(\theta,\rho,\mu)}(x\ast_{(\tilde{\theta},\tilde{\rho},\tilde{\mu})} y)+u\ast_{(\tilde{\theta},\tilde{\rho},\tilde{\mu})}(x\cdot_{(\theta,\rho,\mu)} y)\\
&=&0,
\end{eqnarray*}
we have
\begin{equation*}
\mu(y)\tilde{\rho}(x)-\rho(x)\tilde{\mu}(y)-\mu(y)\tilde{\mu}(x)+\mu(x\ast y)=-\tilde{\mu}(y)\rho(x)+\tilde{\rho}(x)\mu(y)+\tilde{\mu}(y)\mu(x)-\tilde{\mu}(x\cdot y).
\end{equation*}
Thus, $(V,\rho,\mu,\tilde{\rho},\tilde{\mu})$ is a representation of the compatible pre-Lie algebra $(\g,\cdot,\ast)$.
Let $s'$ be another section and $(V,\rho',\mu',\tilde{\rho}',\tilde{\mu}')$ be the corresponding representation of the compatible pre-Lie algebra $(\g,\cdot,\ast)$. Since $s(x)-s'(x)\in V$, then we have
\begin{eqnarray*}
\rho(x)u-\rho'(x)u=(s(x)-s'(x))\cdot_{\hat{\g}}u=0,\\
\mu(x)u-\mu'(x)u=u\cdot_{\hat{\g}}(s(x)-s'(x))=0,
\end{eqnarray*}
which implies that $\rho=\rho'$ and $\mu=\mu'$. Similarly, we have $\tilde{\rho}=\tilde{\rho}'$ and $\tilde{\mu}=\tilde{\mu}'$.
Thus, this representation is independent of the choice of sections.
\end{proof}

\begin{thm}\label{thm:cocycle}
With the above notation, $(\theta,\tilde{\theta})$ is a $2$-cocycle of the compatible pre-Lie algebra $(\g,\cdot,\ast)$ with coefficients in the representation $(V,\rho,\mu,\tilde{\rho},\tilde{\mu})$.
\end{thm}
\begin{proof}
For all $x,y,z\in \g$, by the definition of a pre-Lie algebra, we have
\begin{eqnarray*}0&=&(x\cdot_{(\theta,\rho,\mu)}y)\cdot_{(\theta,\rho,\mu)} z-x\cdot_{(\theta,\rho,\mu)}(y\cdot_{(\theta,\rho,\mu)} z)-(y\cdot_{(\theta,\rho,\mu)} x)\cdot_{(\theta,\rho,\mu)} z+y\cdot_{(\theta,\rho,\mu)} (x\cdot_{(\theta,\rho,\mu)} z)\\
&=&(x\cdot y+\theta(x,y)\big)\cdot_{(\theta,\rho,\mu)} z-x\cdot_{(\theta,\rho,\mu)}\big(y\cdot z +\theta(y,z)\big)\\
&&-\big(y\cdot x+\theta(y,x)\big)\cdot_{(\theta,\rho,\mu)} z+y\cdot_{(\theta,\rho,\mu)}\big(x\cdot z +\theta(x,z))\\
&=&\theta(x\cdot y,z)+\mu(z)\theta(x,y)-\theta(x,y\cdot z)-\rho(x)\theta(y,z)\\
&&-\theta(y\cdot x,z)-\mu(z)\theta(y,x)+\theta(y,x\cdot z)+\rho(y)\theta(x,z)\\
&=&-\partial_{\pi_1+\rho+\mu}\theta(x,y)
\end{eqnarray*}
which implies that $\partial_{\pi_1+\rho+\mu}\theta=0.$ Similarly, we have $\partial_{\pi_2+\tilde{\rho}+\tilde{\mu}}\tilde{\theta}=0.$

For all $x,y,z\in \g$, by \eqref{compatible-pre-Lie}, we have
\begin{eqnarray*}0&=&(x\ast_{(\tilde{\theta},\tilde{\rho},\tilde{\mu})} y)\cdot_{(\theta,\rho,\mu)} z+(x\cdot_{(\theta,\rho,\mu)} y)\ast_{(\tilde{\theta},\tilde{\rho},\tilde{\mu})} z-x\cdot_{(\theta,\rho,\mu)}(y\ast_{(\tilde{\theta},\tilde{\rho},\tilde{\mu})} z)-x\ast_{(\tilde{\theta},\tilde{\rho},\tilde{\mu})}(y\cdot_{(\theta,\rho,\mu)} z)\\
&&-(y\ast_{(\tilde{\theta},\tilde{\rho},\tilde{\mu})} x)\cdot_{(\theta,\rho,\mu)} z-(y\cdot_{(\theta,\rho,\mu)} x)\ast_{(\tilde{\theta},\tilde{\rho},\tilde{\mu})} z+y\cdot_{(\theta,\rho,\mu)}(x\ast_{(\tilde{\theta},\tilde{\rho},\tilde{\mu})} z)+y\ast_{(\tilde{\theta},\tilde{\rho},\tilde{\mu})}(x\cdot_{(\theta,\rho,\mu)} z)\\
&=&(x\ast y+\tilde{\theta}(x,y)\big)\cdot_{(\theta,\rho,\mu)} z+(x\cdot y+\theta(x,y))\ast_{(\tilde{\theta},\tilde{\rho},\tilde{\mu})} z\\
&&-x\cdot_{(\theta,\rho,\mu)}(y\ast z+\tilde{\theta}(y,z))-x\ast_{(\tilde{\theta},\tilde{\rho},\tilde{\mu})}(y\cdot z+\theta(y,z))\\
&&-(y\ast x+\tilde{\theta}(y,x)\big)\cdot_{(\theta,\rho,\mu)} z-(y\cdot x+\theta(y,x))\ast_{(\tilde{\theta},\tilde{\rho},\tilde{\mu})} z\\
&&+y\cdot_{(\theta,\rho,\mu)}(x\ast z+\tilde{\theta}(x,z))+y\ast_{(\tilde{\theta},\tilde{\rho},\tilde{\mu})}(x\cdot z+\theta(x,z))\\
&=&\theta(x\ast y,z)+\mu(z)\tilde{\theta}(x,y)+\tilde{\theta}(x\cdot y,z)+\tilde{\mu}(z)\theta(x,y)\\
&&-\theta(x,y\ast z)-\rho(x)\tilde{\theta}(y,z)-\tilde{\theta}(x,y\cdot z)-\tilde{\rho}(x)\theta(y,z)\\
&&-\theta(y\ast x,z)-\mu(z)\tilde{\theta}(y,x)-\tilde{\theta}(y\cdot x,z)-\tilde{\mu}(z)\theta(y,x)\\
&&+\theta(y,x\ast z)+\rho(y)\tilde{\theta}(x,z)+\tilde{\theta}(y,x\cdot z)+\tilde{\rho}(y)\theta(x,z)\\
&=&-(\partial_{\pi_2+\tilde{\rho}+\tilde{\mu}}\theta+\partial_{\pi_1+\rho+\mu}\tilde{\theta})(x,y,z),
\end{eqnarray*}
which implies that $\partial_{\pi_2+\tilde{\rho}+\tilde{\mu}}\theta+\partial_{\pi_1+\rho+\mu}\tilde{\theta}=0$. Thus, we have
\begin{equation*}
\partial(\theta,\tilde{\theta})=(\partial_{\pi_1+\rho+\mu}\theta,\partial_{\pi_2+\tilde{\rho}+\tilde{\mu}}\theta+\partial_{\pi_1+\rho+\mu}\tilde{\theta},\partial_{\pi_2+\tilde{\rho}+\tilde{\mu}}\tilde{\theta})=0.
\end{equation*}
This finishes the proof.
\end{proof}

\begin{defi}\label{defi:isomorphic}
Let $(\hat{{\g}_1},\cdot_{\hat{{\g}_1}},\ast_{\hat{{\g}_1}})$ and $(\hat{{\g}_2},\cdot_{\hat{{\g}_2}},\ast_{\hat{{\g}_2}})$ be two abelian extensions of a compatible pre-Lie algebra $(\g,\cdot,\ast)$ by  $(V,\cdot_V,\ast_V)$. They are said to be {\bf isomorphic} if there exists a compatible pre-Lie algebra isomorphism $\zeta:(\hat{{\g}_1},\cdot_{\hat{{\g}_1}},\ast_{\hat{{\g}_1}})\longrightarrow (\hat{{\g}_2},\cdot_{\hat{{\g}_2}},\ast_{\hat{{\g}_2}})$ such that the following diagram is commutative:
$$\xymatrix{
  0 \ar[r] &V\ar @{=}[d]\ar[r]^{\iota_1}& \hat{{\g}_1}\ar[d]_{\zeta}\ar[r]^{p_1}&\g\ar @{=}[d]\ar[r]&0\\
     0\ar[r] &V\ar[r]^{ \iota_2} &\hat{{\g}_2}\ar[r]^{p_2} &\g\ar[r]&0.              }$$
\end{defi}

\begin{lem}
Let $(\hat{\g}_1,\cdot_{\hat{\g}_1},\ast_{\hat{\g}_1})$ and $(\hat{\g}_2,\cdot_{\hat{\g}_2},\ast_{\hat{\g}_2})$ be two isomorphic abelian extensions of a compatible pre-Lie algebra $(\g,\cdot,\ast)$ by  $(V,\cdot_V,\ast_V)$. Then they are give rise to the same representation of $(\g,\cdot,\ast)$.
\end{lem}

\begin{proof}
Let $s_1:{\g}_1\longrightarrow \hat{\g}_1$ and $s_2:{\g}_2\longrightarrow \hat{\g}_2$ be two sections of $(\hat{\g}_1,\cdot_{\hat{\g}_1},\ast_{\hat{\g}_1})$ and $(\hat{\g}_2,\cdot_{\hat{\g}_2},\ast_{\hat{\g}_2})$ respectively. By Theorem \ref{thm:representation}, we obtain that $(V,\rho_1,\mu_1,\tilde{\rho}_1,\tilde{\mu}_1)$ and $(V,\rho_2,\mu_2,\tilde{\rho}_2,\tilde{\mu}_2)$ are their representations respectively. Define $s'_1:{\g}_1\longrightarrow \hat{\g}_1$ by $s'_1=\zeta^{-1}\circ s_2$. Since $\zeta:(\hat{\g}_1,\cdot_{\hat{\g}_1},\ast_{\hat{\g}_1})\longrightarrow (\hat{\g}_2,\cdot_{\hat{\g}_2},\ast_{\hat{\g}_2})$ is a compatible pre-Lie algebra isomorphism satisfying the commutative diagram in Definition \ref{defi:isomorphic}, by $p_2\circ \zeta=p_1$, we have
\begin{equation*}
p_1\circ s'_1=p_2\circ \zeta \circ \zeta^{-1}\circ s_2=\Id_{\g}.
\end{equation*}
Thus, we obtain that $s'_1$ is a section of $(\hat{\g}_1,\cdot_{\hat{\g}_1},\ast_{\hat{\g}_1})$. For all $x\in \g, u\in V$, by $\zeta\mid_V=\Id_V$, we have
\begin{eqnarray*}
\rho_1(x)(u)=s'_1(x) \cdot_{\hat{\g}_1} u=(\zeta^{-1}\circ s_2)(x)\cdot_{\hat{\g}_1} u=\zeta^{-1}(s_2(x) \cdot_{\hat{\g}_2} u)=\rho_2(x)(u),\\
\mu_1(x)(u)=u\cdot_{\hat{\g}_1}s'_1(x) =u\cdot_{\hat{\g}_1}(\zeta^{-1}\circ s_2)(x) =\zeta^{-1}(u \cdot_{\hat{\g}_2} s_2(x))=\mu_2(x)(u),
\end{eqnarray*}
which implies that $\rho_1=\rho_2$ and  $\mu_1=\mu_2$.
Similarly, we have $\tilde{\rho}_1=\tilde{\rho}_2$ and $\tilde{\mu}_1=\tilde{\mu}_2$. This finishes the proof.
\end{proof}
So in the sequel, we fixed a representation $(V,\rho,\mu,\tilde{\rho},\tilde{\mu})$ of a compatible pre-Lie algebra $(\g,\cdot,\ast)$ and consider abelian extensions that induce the given representation.

\begin{thm}
Abelian extensions of a compatible pre-Lie algebra $(\g,\cdot,\ast)$ by  $(V,\cdot_V,\ast_V)$ are classified by $\huaH^2(\g;V)$.
\end{thm}

\begin{proof}
Let $(\hat{\g},\cdot_{\hat{\g}},\ast_{\hat{\g}})$ be an abelian extension of a compatible pre-Lie algebra $(\g,\cdot,\ast)$ by  $(V,\cdot_V,\ast_V)$. Choosing a section $s:\g\longrightarrow \hat{\g}$, by Theorem \ref{thm:cocycle}, we obtain that $(\theta,\tilde{\theta})$ is a $2$-cocycle. Now we show that the cohomological class of $(\theta,\tilde{\theta})$ does not depend on the choice of sections. In fact, let $s$ and $s'$ be two different sections. Define $\varphi:\g\longrightarrow V$ by $\varphi(x)=s(x)-s'(x)$. Then for all $x,y\in \g$, we have
\begin{eqnarray*}
 \theta(x,y)&=&s(x)\cdot_{\hat{\g}} s(y)-s(x\cdot y)\\
  &=&\big(s'(x)+\varphi(x)\big)\cdot_{\hat{\g}} \big(s'(y)+\varphi(y)\big)-s'(x\cdot y)-\varphi(x\cdot y)\\
  &=&s'(x)\cdot_{\hat{\g}} s'(y)+\rho(x)\varphi(y)+\mu(y)\varphi(x)-s'(x\cdot y)-\varphi(x\cdot y)\\
  &=&\theta'(x,y)+\partial_{\pi_1+\rho+\mu}\varphi(x,y),
\end{eqnarray*}
which implies that $\theta-\theta'=\partial_{\pi_1+\rho+\mu}\varphi$. Similarly, we have $\tilde{\theta}-\tilde{\theta}'=\partial_{\pi_2+\tilde{\rho}+\tilde{\mu}}\varphi$.

Therefore, we obtain that $(\theta-\theta',\tilde{\theta}-\tilde{\theta}')=\partial \varphi$, $(\theta,\theta')$ and $(\tilde{\theta},\tilde{\theta}')$ are in the same cohomological class.

Now we  prove that isomorphic abelian extensions give rise to the same element in  $\huaH^2(\g;V)$. Assume that $(\hat{{\g}_1},\cdot_{\hat{{\g}_1}},\ast_{\hat{{\g}_1}})$ and $(\hat{{\g}_2},\cdot_{\hat{{\g}_2}},\ast_{\hat{{\g}_2}})$ are two isomorphic abelian extensions of a compatible pre-Lie algebra $(\g,\cdot,\ast)$ by $(V,\cdot_V,\ast_V)$, and $\zeta:(\hat{{\g}_1},\cdot_{\hat{{\g}_1}},\ast_{\hat{{\g}_1}})\longrightarrow (\hat{{\g}_2},\cdot_{\hat{{\g}_2}},\ast_{\hat{{\g}_2}})$ is a compatible pre-Lie algebra isomorphism satisfying the commutative diagram in Definition \ref{defi:isomorphic}. Assume that $s_1:\g\longrightarrow \hat{{\g}_1}$ is a section of $\hat{{\g}_1}$. By $p_2\circ \zeta=p_1$, we have
\begin{equation*}
p_2\circ (\zeta\circ s_1)=p_1\circ s_1=\Id_{\g}.
\end{equation*}
Thus, we obtain that $\zeta\circ s_1$ is a section of $\hat{{\g}_2}$. Define $s_2=\zeta\circ s_1$. Since $\zeta$ is an isomorphism of compatible pre-Lie algebras and $\zeta\mid_V=\Id_V$, for all $x,y\in \g$, we have
\begin{eqnarray*}
 \theta_2(x,y)&=&s_2(x)\cdot_{\hat{{\g}_2}} s_2(y)-s_2(x\cdot y)\\
  &=&(\zeta\circ s_1)(x)\cdot_{\hat{{\g}_2}}(\zeta\circ s_1)(y)-(\zeta\circ s_1)(x\cdot y)\\
  &=&\zeta\big(s_1(x)\cdot_{\hat{{\g}_1}} s_1(y)-s_1(x\cdot y)\big)\\
  &=&\theta_1(x,y),
\end{eqnarray*}
Similarly, we have $\tilde{\theta}_1=\tilde{\theta}_2$.
Thus, isomorphic abelian extensions gives rise to the same element in $\huaH^2(\g;V)$.

Conversely, given two 2-cocycles $(\theta_1,\tilde{\theta}_1)$ and $(\theta_2,\tilde{\theta}_2)$, we can construct two abelian extensions $(\g\oplus V,\cdot_{(\theta_1,\rho,\mu)},\ast_{(\tilde{\theta}_1,\tilde{\rho},\tilde{\mu})})$ and $(\g\oplus V,\cdot_{(\theta_2,\rho,\mu)},\ast_{(\tilde{\theta}_2,\tilde{\rho},\tilde{\mu})})$. If  $(\theta_1,\tilde{\theta}_1), (\theta_2,\tilde{\theta}_2)\in \huaH^2(\g;V)$, then there exists $\varphi:\g\longrightarrow V$, such that $\theta_1=\theta_2+\partial_{\pi_1+\rho+\mu}\varphi$ and $\tilde{\theta}_1=\tilde{\theta}_2+\partial_{\pi_2+\tilde{\rho}+\tilde{\mu}}\varphi$. We define $\zeta:\g\oplus V\longrightarrow \g\oplus V$ by
\begin{equation*}
\zeta(x+u)=x+u+\varphi(x),\quad \forall ~x\in \g, u\in V.
\end{equation*}
For all $x,y\in \g, u,v\in V$, by $\theta_1=\theta_2+\partial_{\pi_1+\rho+\mu}\varphi$, we have
\begin{eqnarray*}&&\zeta\big((x+u)\cdot_{(\theta_1,\rho,\mu)}(y+v)\big)- \zeta(x+u)\cdot_{(\theta_2,\rho,\mu)}\zeta(y+v)\\
 &=&\zeta\big(x\cdot y+\theta_1(x,y)+\rho(x)v+\mu(y)u\big)-\big(x+u+\varphi(x)\big)\cdot_{(\theta_2,\rho,\mu)} \big(y+v+\varphi(y)\big)\\
&=&\theta_1(x,y)+\varphi(x\cdot y)-\theta_2(x,y)-\rho(x)\varphi(y)-\mu(y)\varphi(x)\\
&=&\theta_1(x,y)-\theta_2(x,y)-\partial_{\pi_1+\rho+\mu}\varphi\\
 &=&0,
\end{eqnarray*}
which implies that
\begin{equation}\label{isomorphic-1}
\zeta\big((x+u)\cdot_{(\theta_1,\rho,\mu)}(y+v)\big)=\zeta(x+u)\cdot_{(\theta_2,\rho,\mu)}\zeta(y+v).
\end{equation}
Similarly, we have
\begin{equation}\label{isomorphic-2}
\zeta\big((x+u)\ast_{(\tilde{\theta}_1,\tilde{\rho},\tilde{\mu})}(y+v)\big)=\zeta(x+u)\ast_{(\tilde{\theta}_2,\tilde{\rho},\tilde{\mu})}\zeta(y+v).
\end{equation}
Thus, by \eqref{isomorphic-1} and \eqref{isomorphic-2}, $\zeta$ is a compatible pre-Lie algebra isomorphism from $(\g\oplus V,\cdot_{(\theta_1,\rho,\mu)},\ast_{(\tilde{\theta}_1,\tilde{\rho},\tilde{\mu})})$ to $(\g\oplus V,\cdot_{(\theta_2,\rho,\mu)},\ast_{(\tilde{\theta}_2,\tilde{\rho},\tilde{\mu})})$. Moreover, it is obvious that the diagram in Definition \ref{defi:isomorphic} is commutative. This finishes the proof.
\end{proof}

\noindent{\bf Acknowledgement:}  This work is  supported by  NSF of Jilin Province (No. YDZJ202201ZYTS589), NNSF of China (Nos. 12271085, 12071405) and the Fundamental Research Funds for the Central Universities.

 \end{document}